\documentclass[11pt, a4paper]{amsart}

\usepackage{amsmath, amssymb, amsthm}
\usepackage{chngcntr}
\usepackage[utf8]{inputenc}
\usepackage{tikz}
\usepackage{subcaption}
\usepackage[normalem]{ulem}

\usepackage[numbers,sort&compress]{natbib}

\usepackage[margin=0.75in]{geometry}

\usepackage{fancyhdr}
\usepackage{amsaddr}
\usepackage{color}
\usepackage{enumitem}
\usepackage{dsfont}
\usepackage{cleveref}

\usepackage{xargs}

\newtheorem{theorem}{Theorem}[section]
\newtheorem{proposition}[theorem]{Proposition}
\newtheorem{lemma}[theorem]{Lemma}
\newtheorem{corollary}[theorem]{Corollary}

\newtheorem{remark}[theorem]{Remark}
\newtheorem{definition}[theorem]{Definition}
\newtheorem{example}[theorem]{Example}
\newtheorem{main result}[theorem]{Main Result}

\makeatletter
\newtheorem*{rep@theorem}{\rep@title}
\newcommand{\newreptheorem}[2]{%
\newenvironment{rep#1}[1]{%
 \def\rep@title{#2 \ref{##1}}%
 \begin{rep@theorem}}%
 {\end{rep@theorem}}}
\makeatother
\newreptheorem{theorem}{Theorem}
\newreptheorem{definition}{Definition}

\counterwithin{equation}{section}
\counterwithin{theorem}{section}

\newcommand{\bs}{\boldsymbol}
\newcommand{\mc}{\mathcal}

\newcommand{\R}{\mathbb{R}}
\newcommand{\Cc}{\mathbb{C}}
\newcommand{\N}{\mathbb{N}}
\renewcommand{\P}{\mathbb{P}}

\newcommand{\Lp}[1][p]{\ensuremath{L^{#1}}}
\newcommandx{\Lpq}[2][1=p, 2=q, usedefault=@]{\ensuremath{L^{#1, #2}}}

\newcommand{\C}[1][]{\ensuremath{C^{#1}}}
\renewcommand{\L}{\ensuremath{\mc{L}}}
\newcommandx{\Bsqp}[3][1=\alpha, 2=q, 3=p, usedefault=@]{\ensuremath{B^{#1}_{#3,#2}}}
\newcommandx{\Lipsp}[2][1=\alpha, 2=p, usedefault=@]{\ensuremath{\mathrm{Lip}(#1,#2)}}
\newcommandx{\Wkp}[2][1=\alpha, 2=p, usedefault=@]{\ensuremath{W^{#1, #2}}}
\newcommand{\sobolev}{\ensuremath{\tau}}
\newcommandx{\XY}[2][1=\theta, 2=q, usedefault=@]{\ensuremath{(X, Y)_{#1,#2}}}

\newcommandx{\tbd}[2][1=b, 2=d, usedefault=@]{\ensuremath{t_{#1,#2}}}
\newcommandx{\Tbd}[2][1=b, 2=d,  usedefault=@]{\ensuremath{T_{#1,#2}}}
\newcommand{\Ib}[1][b]{\ensuremath{I_{#1}}}
\newcommand{\set}[1]{\ensuremath{\left\{#1\right\}}}
\newcommand{\tensor}[1]{\ensuremath{\boldsymbol{#1}}}
\newcommandx{\norm}[2][2=]{\ensuremath{\left\| #1 \right\|_{#2}}}
\newcommandx{\snorm}[2][2=]{\ensuremath{\left| #1 \right|_{#2}}}
\newcommandx{\Vbd}[3][1=b, 2=d, usedefault=@]{\ensuremath{V_{#1,#2,#3}}}
\newcommandx{\Vb}[2][1=b, usedefault=@]{\ensuremath{V_{#1, #2}}}
\renewcommand*\d{\mathop{}\!\mathrm{d}}

\newcommand{\linspan}[1]{\ensuremath{\operatorname{span}{\left\{#1\right\}}}}

\newcommand{\TT}[2]{\ensuremath{\mathcal{TT}_{#2}\left(#1\right)}}

\newcommandx{\res}[3][1=d, 2=\bar{d}]{\ensuremath{\mathcal{R}_{#1\rightarrow #2}#3}}
\newcommandx{\ext}[3][1=d, 2=\bar{d}]{\ensuremath{\mathcal{E}_{#1\rightarrow #2}#3}}

\newcommandx{\E}[3][1=n,3=,usedefault=@]{\ensuremath{E_{#1}\left(#2\right)_{#3}}}
\newcommandx{\tool}[1][1=n, usedefault=@]{\ensuremath{\Phi_{#1}}}
\newcommandx{\Asq}[2][1=\alpha, 2=q, usedefault=@]{\ensuremath{A^{#1}_{#2}}}
\newcommand{\rJ}{\ensuremath{r_{\mathrm{J}}}}
\newcommand{\rB}{\ensuremath{r_{\mathrm{B}}}}
\DeclareMathOperator{\cost}{compl}

\renewcommand{\S}[3]{\ensuremath{\mathcal{S}^{#1,#2}_{#3}}}
\newcommand{\xk}[1][N]{\ensuremath{(x_k)_{k=0}^{#1}}}
\newcommand{\xkb}[1][N]{\ensuremath{({x^b_k)_{k=0}^{#1}}}}
\newcommand{\Sfree}[2]{\ensuremath{\mathcal{S}_{\mathrm{fr}}^{#1,#2}}}
\newcommand{\Sfreeb}[2]{\ensuremath{\mathcal{S}^{b,#1,#2}_{\mathrm{fr}}}}
\newcommand{\indicator}[1]{\ensuremath{\mathds{1}_{#1}}}
\newcommand{\m}{\ensuremath{\bar{m}}}
\newcommandx{\interpol}[2][1=b, 2=d, usedefault=@]{\ensuremath{\mathcal{I}_{#1,#2}}}
\newcommandx{\li}[4][2=b, 3=l, 4=j, usedefault=@]{\ensuremath{#1_{#2,#3,#4}}}

\newcommand{\diff}{\ensuremath{\Delta}}
\renewcommand{\mod}{\ensuremath{\omega}}

\newcommand{\w}{\ensuremath{\mathrm{w}}}
\newcommand{\setf}{\ensuremath{\Omega}}
\newcommand{\Pch}{\ensuremath{P_{\mathrm{Ch}}}}
\newcommand{\PhinB}{\ensuremath{\Phi_n^{\mathrm{B}}}}
\newcommand{\CB}{\ensuremath{c_{\mathrm{B}}}}

\newcommand{\cont}{\ensuremath{\mathfrak{{c}}}}
\DeclareMathOperator*{\dist}{dist}

\renewcommand{\hbar}{\ensuremath{\bar{h}}}

\title{Approximation Theory of Tree Tensor Networks:
Tensorized Univariate Functions -- Part II}

\author{Mazen Ali \and Anthony Nouy}
\address{Fraunhofer ITWM, 67663 Kaiserslautern, Germany}
\address{Centrale Nantes, Nantes Université, LMJL UMR CNRS 6629, France}
\email{mazen.ali@itwm.fraunhofer.de}
\email{anthony.nouy@ec-nantes.fr}

\thanks{Acknowledgments: The authors acknowledge AIRBUS Group for the financial support with the project AtRandom.}
\date{\today}


\begin{document}

\begin{abstract}
    We study the approximation by tensor networks (TNs) of functions from  classical smoothness classes. 
    The considered approximation tool combines a tensorization of 
     functions in $\Lp([0,1))$, which allows to identify a univariate function with a multivariate function (or tensor), and the use of tree tensor networks (the tensor train format) for exploiting low-rank structures of  multivariate functions. The resulting tool can be interpreted as a feed-forward neural network, with first layers implementing the tensorization, interpreted as a particular featuring step, followed by a sum-product network with sparse architecture. 
  
    In part I of this work, we presented several approximation classes associated with different 
    measures of complexity of tensor networks and studied their properties.  
    
    In this work (part II), we show how classical
    approximation tools, such as polynomials or splines (with fixed or free knots), can be
    encoded as a tensor network with controlled complexity.
    We use this to derive direct (Jackson) inequalities
    for the approximation spaces of tensor networks.
    This is then utilized to show that Besov spaces are
    continuously embedded into these approximation spaces.
    In other words, we show that arbitrary Besov functions
    can be approximated with optimal or near to optimal rate.  
    We also show that an arbitrary function in
    the approximation class possesses no Besov smoothness,
    unless one limits the \emph{depth} of the tensor network. 
\end{abstract}

\maketitle

\section{Introduction}

Approximation of functions is an integral part of mathematics
with many important applications in various other fields
of science, engineering and economics.
Many classical approximation methods -- such
as approximation with polynomials, splines, wavelets,
rational functions, etc.\ -- are by now thoroughly understood.

In recent decades new families of methods have gained
increased popularity due to their success in various applications
-- tensor and neural networks (TNs and NNs). See, e.g., 
\cite{Bachmayr2016, nouy:2017_morbook, cichocki2016tensor1, cichocki2017tensor2, Orus2019, Haykin2009,goodfellow2016deep}
and references therein for an overview.  

In part I of this work (see \cite{partI}), we defined
approximation classes of tensor networks
and studied
their properties.
Our goal was to introduce a new method for
analyzing the expressivity of tensor networks, 
as was done in \cite{Gribonval2019} for
rectified linear unit (ReLU) and
rectified power unit (RePU) neural networks.

In this work (part II), we continue with our endeavor by
studying how these approximation classes of tensor networks
are related to the well-known Besov spaces
(see, e.g., \cite{Triebel1992}).
In particular, we will show that any Besov function can be
approximated with optimal rate with a tensor network.
On the other hand, an arbitrary function from the
approximation class of tensor networks has no Besov smoothness,
unless we restrict the \emph{depth} of the tensor network.
This illustrates the high expressivity of
(deep) tensor networks. Interestingly, similar results
were shown for deep ReLU and RePU networks
in, e.g., \cite{Gribonval2019, yarotsky2018optimal, Opschoor2020, Opschoor201907}. 
We also consider the case of analytic functions and show that they can be approximated with exponential rate.
\\

The outline is as follows. In \Cref{sec:recalling}, 
we begin by recalling some notations and results from part I \cite{partI}.
We then state the main results of this work in \Cref{sec:main-results}.
In \Cref{sec:encoding},
we discuss how classical approximation tools can be encoded with tensor networks
and estimate the resulting complexity.
Among the classical tools considered are
fixed knot splines, free knot splines,
polynomials (of higher order) and
general multi-resolution analysis (MRA).
In \Cref{sec:dirinv},
utilizing results from the previous section,
we show direct estimates for our approximation tool
that lead to the embeddings from
\Cref{thm:introembed}.
We also show that inverse embeddings can only hold if we restrict
the depth of the tensor networks.
We conclude on \Cref{sec:conclusion} by a brief discussion of depth
and sparse connectivity.

For a brief introduction to TNs and NNs, we
refer to the introduction in \cite{partI}.


\section{Recalling Results of Part I}\label{sec:recalling}


\subsection{Tensorization}
We consider one-dimensional functions on the unit interval
\begin{equation*}
    f:[0,1)\rightarrow\R.
\end{equation*}
We introduce a uniform partition of $[0,1)$ with $ b^d$ intervals $[x_i,x_{i+1})$ with $x_i = b^{-d} i$, $0 \le i \le b^d$, with base $b = 2,3, ...$, and
exponent $d\in\N_0$. Any $x\in[0,1)$ falls either on or in-between one of
the points $x_i$. Thus, using a $b$-adic expansion of the integer $i$,
we can define a \emph{conversion map} $\tbd$ as
\begin{align*}
x=t_{b,d}(i_1,\ldots,i_d,y) := \sum_{k=1}^di_kb^{-k}+b^{-d}y.
\end{align*}
for $y\in[0,1)$ and $i_k \in \set{0,\ldots,b-1} := \Ib$.

\begin{definition}[Tensorization Map]\label{def-tensorization-map}
For a base $b\in \N$ ($b\ge 2$) and a level $d\in \N,$ we define the 
 \emph{tensorization map} 
	\begin{align*}
	    \Tbd:\R^{[0,1)}\rightarrow\R^{\Ib^d\times[0,1)}, \quad f \mapsto f \circ t_{b,d} := \tensor{f}
	\end{align*}
which associates to a function 
$f \in \R^{[0,1)}$ the multivariate function 
	$	   \tensor{f}\in \R^{\Ib^d\times[0,1)} $ such that 
	\begin{align*}
	    \tensor{f}(i_1,\ldots,i_d,y):=f(\tbd(i_1,\ldots,i_d,y)).
	\end{align*}	
\end{definition}
For $i=\sum_{k=1}^{d} i_k b^{d-k}$, the partial evaluation $\tensor{f}(i_1,\ldots,i_d,\cdot)$ of $\tensor{f}$ is equal to the function $f(b^{-d}(i + \cdot)) \in \R^{[0,1)}$, which is the restriction of $f$ to the interval $[b^{-d}i,b^{-d}(i+1))$  rescaled to 
$[0, 1)$.

The space $\R^{\Ib^d\times[0,1)}$ can be identified with the algebraic tensor space $$ \mathbf{V}_{b,d} := \R^{\Ib^d} \otimes \R^{[0,1)}=  \underbrace{\R^{\Ib} \otimes \hdots \otimes \R^{\Ib}}_{\text{$d$ times }}\otimes \R^{[0,1)} =: (\R^{\Ib})^{\otimes d} \otimes \R^{[0,1)} ,$$ which is the set of functions $\tensor{f}$ defined on $\Ib^d \times [0,1)$ that admit 
a representation
\begin{equation}\label{algebraic-tensor}
\tensor{f}(i_1,\hdots,i_d,y) = \sum_{k=1}^r v_1^k(i_1) \hdots v_d^k(i_d) g^k(y) :=\sum_{k=1}^r (v_1^k \otimes \hdots \otimes v_d^k \otimes g^k)(i_1,\hdots,i_d,y)    
\end{equation}
for some $r \in \N$ and for some functions $v_\nu^k \in \R^{\Ib}$ and $g^k \in \R^{[0,1)}$, $1\le k \le r$, $1\le \nu\le d$.

%

We showed that $\Lp$ functions can be isometrically identified with
tensors.
\begin{theorem}[Tensorization is an $\Lp$-Isometry]\label{thm:tensorizationmap}
For any $0 < p \le \infty$, define
$$\mathbf{V}_{b,d,L^p} := {\R^{\Ib}}^{\otimes d} \otimes L^p([0,1)) \subset \mathbf{V}_{b,d}.$$
Then, $\Tbd$ is a linear isometry from $\Lp([0,1))$ to  $\mathbf{V}_{b,d,L^p}$ equipped with the (quasi-)norm 
  $\Vert \cdot\Vert_p$ defined by 
 $$\norm{\tensor{f}}[p]^p = 
 \sum_{j_1\in \Ib} \hdots \sum_{j_d\in \Ib} b^{-d} \Vert  \tensor{f}(j_1,\hdots,j_d,\cdot) \Vert_{p}^p
 $$
 for $p<\infty$,  
or  
$$
\norm{\tensor{f}}[\infty] = \max_{j_1\in \Ib} \hdots \max_{j_d \in \Ib} \Vert \tensor{f}(j_1,\hdots,j_d,\cdot) \Vert_\infty.
$$
\end{theorem}   
For an interpretation of the tensorization as a particular featuring step and its implementation as a specific neural network,  we refer to \cite[Section 4]{partI}.


\subsection{Ranks}
The minimal $r\in\N$ such that $\tensor{f}$ admits a representation
as in \eqref{algebraic-tensor}
is referred to as the \emph{canonical tensor rank} of $\tensor{f}$.
Crucial for this work is the following notion of multi-linear rank.

\begin{definition}[$\beta$-rank]\label{def:betarank}
For $\beta \subset \{1,\hdots,d+1\}$, the  $\beta$-rank of $\tensor{f} \in \mathbf{V}_{b,d}$, denoted $r_\beta(\tensor{f})$, is the minimal integer  such that $\tensor{f}$ admits a representation of the form 
 \begin{align}\label{eq:tensorfrep}
        \tensor{f}=\sum_{k=1}^{r_\beta(\tensor{f})}
        \bs{v}^k_\beta\otimes\bs{v}^k_{\beta^c},
    \end{align}
    where $\bs{v}_\beta^k \in \mathbf{V}_\beta$ and 
     $\bs{v}_{\beta^c}^k \in \mathbf{V}_{\beta^c}$.  
\end{definition}

Since a function $f\in\R^{[0,1)}$ admits a representation as a tensor
for different $d\in\N$, we require the following notion of ranks for a function.

\begin{definition}[$(\beta,d)$-rank]\label{def:betadrank}
For a function $f \in \R^{[0,1)}$, $d\in \N$ and $\beta \subset \{1,\hdots,d+1\} $, we define the $(\beta,d)$-rank of $f$, denoted $r_{\beta,d}(f)$, as the rank of its tensorization in $\mathbf{V}_{b,d}$, 
$$
r_{\beta,d}(f) = r_\beta(\Tbd f).
$$
\end{definition}
We mostly consider subsets $\beta$ of the form $\{1,\hdots,\nu\}$ or $\{\nu+1,\hdots,d+1\}$ for some $\nu\in \{1,\hdots,d\}$,
and thus we will use the shorthand notations
$$
r_{\nu}(\tensor{f}) := r_{\{1,\hdots,\nu\}}(\tensor{f}), \quad r_{\nu,d}(f) =  r_{\{1,\hdots,\nu\},d}(f).
$$

The notion of partial evaluations will be useful for estimating
the ranks of a function $f$.
\begin{lemma}\label{link-ranks-d-nu}
Let $f \in \R^{[0,1)}$ and $d\in \N$. For any $1\le \nu\le d,$
$$
r_{\nu,d}(f) = r_{\nu,\nu}(f)$$
and 
$$
r_{\nu,\nu}(f) = \dim\mathrm{span}\{f( b^{-\nu}(j+\cdot)) : 0\le j \le b^\nu-1\} .
$$
\end{lemma}	


\subsection{Tensor Subspaces}
For a subspace $S\subset \R^{[0,1)}$, we define the tensor subspace
\begin{align*}
 \mathbf{V}_{b,d,S} :=  (\R^{\Ib})^{\otimes d} \otimes S \subset  \mathbf{V}_{b,d}, 
\end{align*}
and the corresponding subspace of functions,
$$
\Vbd{S} = \Tbd^{-1}(\mathbf{V}_{b,d,S}).
$$
We frequently use $S=\P_m$ where $\P_m$ is the space of polynomials of degree up to $m\in\N_0$, with the shorthand notation
\begin{align}\label{def:vbdm}
    \Vbd{m}:=\Vbd{\P_m}.
\end{align}

It will be important in the sequel that $S$ satisfies the following property,
analogous to properties satisfied by spaces generated by
multi-resolution analysis (MRA).

\begin{definition}[Closed under Dilation]\label{def:badicdil}
We say that a linear space $S$ is closed under $b$-adic dilation if for any $f \in S$ and any $k \in \{0,\hdots,b-1\}$, 
$$
f(b^{-1}(\cdot + k)) \in S.
$$
\end{definition}

If $S$ is closed under $b$-adic dilation, we can show that
\begin{align*}
S := V_{b,0,S} \subset   V_{b,1,S} \subset \hdots \subset  V_{b,d,S} \subset \hdots,
\end{align*}
and that
\begin{align*}
    \Vb{S}:=\bigcup_{d\in\N}\Vbd{S},
\end{align*}
is a linear subspace of $\R^{[0,1)}$. Moreover, we have

\begin{lemma}\label{rank-bound-VbdS-closed-dilation} 
Let $S$ be closed under $b$-adic dilation.
\begin{enumerate}
\item[(i)] If $f \in S$, then for any $d\in \N$, $f\in V_{b,d,S}$ satisfies 
$$
r_{\nu,d}(f) \le \min\{b^\nu , \dim S\}, \quad 1\le \nu\le d.
$$
\item[(ii)] If $f\in V_{b,d,S}$, then for any $\bar d \ge d$, $f \in V_{b,\bar d,S}$ satisfies 
      \begin{align}
       & r_{\nu,\bar d}(f) = r_{\nu,d}(f) \leq \min\set{b^\nu, b^{d-\nu} \dim S},\quad  1\le \nu \le d,\notag\\
      &  r_{\nu,\bar d}(f)\leq \min\set{b^{\nu}  ,\, \dim S}, \quad d < \nu \le \bar d\label{eq:rjbig}.
    \end{align}
    \end{enumerate}
\end{lemma}

Letting $\mathcal{I}_S$ be a linear projection operator from $\Lp([0,1))$ to a finite-dimensional space $S$, we 
define a linear operator $\mathcal{I}_{b,d,S} $ from  $\Lp([0,1))$ to $V_{b,d,S}$ 
by 
\begin{equation*}
	(\mathcal{I}_{b,d,S} f) (b^{-d} (j +\cdot))=  \mathcal{I}_{S} (f(b^{-d} (j+\cdot ))), \quad 0\le j < b^d,\label{local-projection}
\end{equation*}
or \cite[Lemma 2.29]{partI}
\begin{align}
 \mathcal{I}_{b,d,S}  = T_{b,d}^{-1}  \circ  (id_{\{1,\hdots,d\}} \otimes \mathcal{I}_S) \circ T_{b,d}. \label{tensor-structure-local-projection}
\end{align}
We can bound the ranks of local projections onto $S$
in the following sense.
\begin{lemma}[Local projection ranks]\label{local-projection-ranks}
For any $f\in \Lp$, $\mathcal{I}_{b,d,S} f \in V_{b,d,S}$ satisfies 
$$
r_{\nu,d}(\mathcal{I}_{b,d,S} f) \le r_{\nu,d}(f) ,\quad 1\le \nu\le d.
$$
\end{lemma}


\subsection{Tensor Train Format and Corresponding Approximation Tool}

In this work, we will be using tree tensor networks, and more precisely the \emph{tensor train format}, to define our approximation tool. 
\begin{definition}[Tensor Train Format]\label{def:ttfuncs}
	The set of tensors in $\mathbf{V}_{b,d}$ in tensor train (TT) format 
	with ranks at most $\bs r:=(r_\nu)_{\nu=1}^d$  is defined as
	\begin{align*}
	    \TT{\mathbf{V}_{b,d,S}}{\bs r}:=\set{\tensor{f} \in\mathbf{V}_{b,d,S}:\;r_\nu(\tensor{f})\leq r_\nu , \; 1 \le \nu\le d}.
	\end{align*}
	This defines a set of univariate functions 
	$$\Phi_{b,d,S,\bs r}= T_{b,d}^{-1} ( \TT{\mathbf{V}_{b,d,S}}{\bs r}) = \{f \in V_{b,d,S} : r_{\nu}(f) \le r_\nu , 1\le \nu \le d\},$$
	where $r_\nu(f) := r_{\nu,d}(f)$, that we further call tensor train format for univariate functions.
\end{definition}
Letting $\{\varphi_k\}_{k=1}^{\dim S}$ be a basis of $S$, a tensor $\tensor{f} \in\TT{\mathbf{V}_{b,d,S}}{\bs r}$ admits a representation
\begin{align}
    \tensor{f}(i_1,\ldots,i_d, y)&=\sum_{k_1=1}^{r_1}\cdots\sum_{k_d=1}^{r_d} \sum_{k_{d+1}}^{\dim S} v_1^{k_1}(i_1) v_2^{k_1,k_2}(i_2)
    v_3^{k_2,k_3}(i_3)\cdots v_d^{k_{d-1},k_d}(i_d) v_{d+1}^{k_d,k_{d+1}} \varphi_k(y),
\label{repres-tt}
    \end{align}
where  the parameters $\mathbf{v}:= (v_1 ,\hdots,v_{d+1}) $ form a   tensor network  with 
$$
 \mathbf{v} := (v_1 ,\hdots,v_{d+1}) \in \R^{b\times r_1} \times \R^{b\times  r_1 \times r_2} \times \hdots  \times \R^{b\times  r_{d-1} \times r_d} \times \R^{r_d \times \dim S} := \mathcal{P}_{b,d,S,\bs r}. 
$$
Denoting by $\mathcal{R}_{b,d,S,\bs r}(\mathbf{v})$ the map  which associates to a tensor network $\mathbf{v}$ the function $f = T_{b,d} \tensor f$ with $\tensor f$ defined by \eqref{repres-tt}, we have 
	$$\Phi_{b,d,S,\bs r} = \{  \varphi = \mathcal{R}_{b,d,S,\bs r}(\mathbf{v})
	: \mathbf{v} \in  \mathcal{P}_{b,d,S,\bs r} \}.$$
We introduced three different measures
of complexity of a tensor network $\mathbf{v} \in \mathcal{P}_{b,d,S,\bs r}$:
\begin{align}\label{eq:complex}    
   \cost_{\mathcal{N}}(\mathbf{v})  &:=  
  \sum_{\nu=1}^{d} r_\nu,\\  
   \cost_\mathcal{C}(\mathbf{v})  &:=
  br_1+b\sum_{k=2}^{d}r_{k-1}  r_k+
    r_d \dim S,\notag \\    
    \cost_\mathcal{S}(\mathbf{v}) &:=  \sum_{\nu=1}^{d+1} \Vert v_\nu \Vert_{\ell_0} \notag,
\end{align}
where $\Vert v_\nu\Vert_{\ell_0}$ is the number of non-zero entries in the tensor $v_\nu$.  
The function
$\cost_\mathcal{C}(\mathbf{v})$ is a natural measure of complexity equal to the number of parameters. The function
$\cost_\mathcal{S}$ is also a natural measure of complexity taking account the sparsity of tensors in the tensor network, and equal to the number of non-zero parameters. 
When interpreting a tensor network $\mathbf{v}$ as a sum-product neural network,  $\cost_\mathcal{N}(\mathbf{v})$ corresponds to the number of neurons, $\cost_\mathcal{C}(\mathbf{v})$ to the number of weights, and $\cost_\mathcal{S}(\mathbf{v})$ the number of non-zero weights (taking into account sparsity in the connectivity of the network). 

Our approximation tool for univariate functions is defined as
 $$
\tool[] := (\Phi_n)_{n\in \N}, \quad  \Phi_n = \{\varphi \in \Phi_{b,d,S,\bs r} : d\in \N, \bs r \in \N^d , \cost(\varphi) \le n\},
 $$
 where $\cost(\varphi)$ is a measure of complexity of a function $\varphi$, defined as 
 $$
 \cost(\varphi) := \min \{\cost(\mathbf{v}) :  \mathcal{R}_{b,d,S,\bs r}(\mathbf{v}) = \varphi, d \in \N , \bs r\in \N^d\},
 $$ 
 where the infimum is taken over all tensor networks $\mathbf{v}$ whose realization is  the function  $\varphi$. 
The three complexity measures  define three types of subsets
\begin{align*}
    \tool^{\mc N}&:=\set{\varphi\in\tool[]:\;\cost_{\mc N}(\varphi)\leq n},\\
        \tool^{\mc C}&:=\set{\varphi\in\tool[]:\;\cost_{\mc C}(\varphi)\leq n},\notag\\
            \tool^{\mc S}&:=\set{\varphi\in\tool[]:\;\cost_{\mc S}(\varphi)\leq n}.\notag
\end{align*}
For the implementation of the resulting approximation tool as a particular feed-forward neural network, we refer to \cite[Section 4]{partI}.

\subsection{Approximation Spaces}\label{sec:appspacestensors}


Let $X$ be a quasi-normed linear space, $\tool\subset X$  subsets of $X$
for $n\in\N_0$ and $\tool[]:= (\tool)_{n\in \N_0}$ an approximation tool.
We define the best approximation error
\begin{align*}
E_n(f):= E(f,\Phi_n):=\inf_{\varphi\in\tool}\norm{f-\varphi}[X].
\end{align*}
The approximation classes $\Asq$ of $\tool[] = (\tool)_{n\in \N_0}$ are defined by  
    \begin{align*}
        \Asq:=\Asq(X):=\Asq(X,\tool[]):=\set{f\in X:\; \norm{f}[\Asq]<\infty},
    \end{align*}
    for $\alpha>0$ and $0<q\le \infty$, with 
    \begin{align*}
        \norm{f}[\Asq]:=
        \begin{cases}
            \left(\sum_{n=1}^\infty[n^\alpha E_{n-1}({f})]^q\frac{1}{n}\right)^{1/q},&\quad 0<q<\infty,\\
            \sup_{n\geq 1}[n^\alpha E_{n-1}({f})],&\quad q=\infty.
        \end{cases}
    \end{align*}
Considering $\tool \in \{\tool^{\mathcal{N}} , \tool^{\mathcal{C}}, \tool^{\mathcal{S}} \}$, we obtain families of 
approximation classes of tensor networks
\begin{align}\label{eq:approxclasses}
N_q^\alpha(X) &:= A_q^\alpha(X , (\tool^{\mathcal{N}})_{n\in \N}),\\
C_q^\alpha(X) &:= A_q^\alpha(X ,  (\tool^{\mathcal{C}})_{n\in \N}),\notag\\
S_q^\alpha(X) &:= A_q^\alpha(X ,  (\tool^{\mathcal{S}})_{n\in \N}).\notag
\end{align}
There are a few properties of $\tool[]$ that turn out to be
crucial, if one is to obtain strong statements about the properties of the associated approximation classes.
\begin{enumerate}[label=(P\arabic*)]
    \item\label{P1}    $0\in\tool$, $\tool[0]=\set{0}$.
    \item\label{P2}     $\tool[n]\subset\tool[n+1]$.
    \item\label{P3}     $a\tool=\tool$ for any $a\in\R\setminus\set{0}$.
    \item\label{P4}     $\tool+\tool\subset\tool[cn]$ for some $c:=c(\tool[])$.
    \item\label{P5}     $\bigcup_{n\in \N_0} \tool$ is dense in $X$.
    \item\label{P6}     $\tool$ is proximinal in $X$, i.e. each $f\in X$ has a best approximation in $\tool$.
\end{enumerate}
In particular, \ref{P1}--\ref{P4} together with a Jackson inequality imply that the approximation classes are  quasi-normed linear spaces. 
A main result of Part I \cite{partI} is
\begin{theorem}\label{comparing-spaces}
For any $\alpha>0$, $0<p\leq\infty$ and $0<q \le \infty$,
the approximation classes $N_q^\alpha(\Lp)$, $C_q^\alpha(X)$
and $S_q^\alpha(X)$ are quasi-normed linear spaces
satisfying the continuous embeddings
\begin{align*}
&C^{\alpha}_q(L^p) \hookrightarrow S^{\alpha}_q(L^p) \hookrightarrow N^{\alpha}_q(L^p)\hookrightarrow   C^{\alpha/2}_q(L^p).
\end{align*}
\end{theorem}


\section{Main Results of This Work}\label{sec:main-results}
The main results of this work are \Cref{thm:sobolev}, 
\Cref{cor:besovspaces}, \Cref{thm:analytic} and \Cref{thm:noinverse},
and can be summarized as follows.

\begin{main result}[Direct Embedding for Sobolev Spaces]\label{thm:introembed}
Let $\Wkp[r]$ denote the
Sobolev space of $r\in\N$ times weakly differentiable,
$p$-integrable functions
and
$\Bsqp[\alpha]$ the Besov space of smoothness $\alpha>0$,
with primary parameter $p$
and secondary parameter $q$.
Then, for $S=\P_m$ with a fixed $m\in\N_0$,
we show that 
for $1\leq p\leq\infty$ and any $r\in\N$
    \begin{align*}
        W^{r,p} \hookrightarrow
        N^{2r}_\infty(\Lp),\quad
        W^{r,p} \hookrightarrow
        C^{r}_\infty(\Lp)\hookrightarrow S^{r}_\infty(\Lp),
    \end{align*}
    and for $0<q\leq\infty$ and any $\alpha>0$
    \begin{align*}
        \Bsqp[\alpha]\hookrightarrow N^{2\alpha}_q(\Lp),\quad
        \Bsqp[\alpha]\hookrightarrow
        C^{\alpha}_q(\Lp)\hookrightarrow        
        S^{\alpha}_q(\Lp).
    \end{align*}
    
    Note that for $p=q$ and non-integer $\alpha>0$,
    $\Bsqp[\alpha][p][p]=W^{\alpha,p}$ is the fractional Sobolev space.
    Moreover, these results
    can also be extended to the range $0<p<1$,
    see \Cref{rem:extendp}.
    \end{main result}

\begin{main result}[Direct Embedding for Besov Spaces]\label{thm:introembedbesov}
Let $\Bsqp[\alpha]$ denote the Besov space of smoothness
$\alpha>0$,
with primary parameter $p$ and secondary parameter $q$.
Then, for $S=\P_m$ with a fixed $m\in\N_0$,
we show that
for any $1\leq p<\infty$,
any $0<\tau<p$, any $r>1/\tau-1/p$ and any $0<\bar r<r$,
    \begin{align*}
            \Bsqp[r][\sobolev][\sobolev]
            \hookrightarrow N^{\bar r}_\infty(\Lp)
            \hookrightarrow C^{\bar r/2}_\infty(\Lp),\quad
            \Bsqp[r][\sobolev][\sobolev]
            \hookrightarrow S^{\bar r}_\infty(\Lp),
    \end{align*}    
    and for any $0<q\leq\infty$, any $0<\alpha<\bar r$
    \begin{align*}
       (\Lp, \Bsqp[r][\sobolev][\sobolev])_{\alpha/\bar r,q}
       \hookrightarrow
       N_q^{\alpha}(\Lp)
       \hookrightarrow
       C_q^{\alpha/2}(\Lp),\quad
        (\Lp, \Bsqp[r][\sobolev][\sobolev])_{\alpha/\bar r,q}   
       \hookrightarrow
       S_q^{\alpha}(\Lp),
    \end{align*}
    where $(X,Y)_{\theta,q}$, $0<\theta<1$ is the real $K$-interpolation
    space between $X$ and $Y\hookrightarrow X$.   
    \end{main result}

\begin{remark}
    Note that both \Cref{thm:introembed} and \Cref{thm:introembedbesov}
    apply to Besov spaces. The Besov spaces in
    \Cref{thm:introembed} are of the type $\Bsqp$, where $p$ is the same for
    the error measure. Such Besov spaces
    are captured by linear approximation and for $p\geq 1$ these
    are equal to or are very close to Sobolev spaces.
    
    On the other hand, the Besov spaces $\Bsqp[\alpha][\sobolev][\sobolev]$
    for $1/\sobolev=\alpha+1/p$ are much larger and these correspond
    to the critical embedding line, see also \Cref{fig:DeVore}
    of \Cref{sec:dirinv}.
    These Besov spaces can only be captured by non-linear approximation.
    Our results require $\alpha>1/\sobolev-1/p$, i.e., Besov spaces that
    are strictly
    above the critical line.
\end{remark}

 \begin{main result}[No Inverse Embedding]
   Let $\Bsqp[\alpha]$ denote the Besov space of smoothness $\alpha$,
with primary parameter $p$ and secondary parameter $q$.
We show that for any $\alpha>0$, $0<p,q\leq\infty$,
    and any $\tilde{\alpha}>0$,
    \begin{align*}
        C^\alpha_q(\Lp)\not\hookrightarrow\Bsqp[\tilde{\alpha}].
    \end{align*}
\end{main result}

\begin{main result}[Spectral Approximation]\label{thm:introanalytic}
For $S=\P_m$ with a fixed $m\in\N_0$,
we show that if $f$ is analytic on an open domain containing $[0,1]$,  
    \begin{align*}
      E_n^{\mc N}(f)_\infty&\leq C\rho^{-n^{1/2}},\\
      E_n^{\mc S}(f)_\infty&\leq E_n^{\mc C}(f)_\infty\leq C\rho^{-n^{1/3}},
    \end{align*}
    for constants $C,\rho>1$. This can be extended
     to analytic functions with singularities using ideas from
     \cite{Kazeev2017}.
\end{main result}

In words:
\begin{itemize}
    \item    For the approximation tools $\tool^{\mc C}$ and $\tool^{\mc S}$
                (of fixed polynomial degree $m \in \N_0$), we obtain
                optimal approximation rates for Sobolev spaces
                $W^{r,p}$ of \emph{any} order $r>0$.
                
    \item     For the approximation tool $\tool^{\mc N}$,
                we obtain twice the optimal rate.
                Note, however, that the corresponding complexity measure
                only reflects the number of neurons in a corresponding
                tensor network. It does not reflect the representation or
                computational complexity.
                Moreover, from \cite{DeVore1989}
                we know that an optimal approximation tool with \emph{continuous} parametrization
                 for the Sobolev space $W^{r,p}$
                cannot exceed the rate $r$,
                see also \cite{yarotsky2018optimal}.
    \item    For the approximation tools $\tool^{\mc N}$ and
               $\tool^{\mc S}$, we obtain
               \emph{near to} optimal
               rates\footnote{I.e., the approximation rates are arbitrary close to optimal.} for the Besov space $B^r_{\tau,\tau}$,
               for any order $r>0$.
               For $\tool^{\mc C}$, the approximation rate is
               near to half the optimal rate.
    \item    Particularly the tool $\tool^{\mc S}$ is interesting, as it
                corresponds to deep, sparsely connected networks. The above
                results imply that deep, sparsely connected tensor networks
                can optimally replicate both $h$-uniform and
                $h$-adaptive approximation
                of any order.
    \item    All approximation tools achieve exponential approximation rates
                for analytic
                target functions. Together with the previous result,
                this implies that deep, sparsely connected tensor networks can
                optimally replicate $hp$-adaptive 
                approximation\footnote{Even though the underlying polynomial
                degree of the tensor network remains fixed.}.
    \item    Finally, an arbitrary function from any of the three approximation
                classes possesses no Besov smoothness. We will also
                see in \Cref{prop:inverse} that this can be mainly attributed to the
                depth of the tensor network.
\end{itemize}

We restrict ourselves in this work to
approximation of functions on intervals in one dimension to
focus on the presentation of the basic concepts and avoid
the technical difficulties of general multi-dimensional domains.
However, the ideas can be in principle extended to any dimension
and we intend to do
so in a part III of this work.

We base our approximation tool on the TT format.
Although some of our results would remain unchanged for other tree-based tensor formats, ranks are
generally affected by the choice of the format.
This is known for multi-dimensional approximation with tensor formats,
see, e.g., \cite{Buczynska2015, Buczynska2020}.
In the multi-dimensional case ranks remain low if the format ``fits'' the problem at hand, in a certain sense.
E.g., if the format mimics the interaction structure dictated by the differential operator, see \cite{Ali2019}.
In the context of tensorized 1D approximation, the tensor format would have to fit the self-similarity, periodicity or other algebraic
features of the target function,
see \cite{Grasedyck2010, Oseledets2012, Guilmard2018}.

We thus stress the following point concerning the approximation
power of tree-based tensor networks:
on one hand, when comparing approximation classes of different tensor networks to
spaces of classical smoothness -- the distinction between different tree-based formats seems insignificant. On the other hand,
when comparing approximation classes of different tensor networks
\emph{to each other} -- we expect these to be substantially different. 


\section{Encoding Classical Approximation Tools}\label{sec:encoding}

In this section, we demonstrate how classical approximation tools can
be represented with tensor networks and bound the complexity of such
a representation.
Specifically, we consider representing fixed knot splines,
free knot splines, polynomials
and multi-resolution analysis (MRAs).
This will be the basis for Section \ref{sec:dirinv} where we use these complexity
estimates to prove embeddings of a scale of interpolation spaces into the approximation classes $N^\alpha_q(\Lp)$, $C^\alpha_q(\Lp)$
and $S^\alpha_q(\Lp)$.
Our background space is as before $\Lp$, where we specify the range
of $p$ where necessary.


\subsection{Polynomials}
Let us first consider the encoding of a polynomial of degree $\bar m$ in $\Vbd{\bar m}$. 
\begin{lemma}[Ranks for Polynomials]\label{ranks-polynomials}
Let $\varphi\in \P_{\m}$, $\m\in \N$. For any $d\in \N$, $\varphi\in\Vbd{\m}$ and for $1\leq \nu\leq d$,
$$
r_{\nu,d}(\varphi) \le \min\{\m+1,b^\nu\}.
$$
\end{lemma}
\begin{proof}
Since $\P_{\m}$ is closed under $b$-adic dilation, the result simply follows from 
\Cref{rank-bound-VbdS-closed-dilation} (i).
\end{proof}
Now we consider the representation of a function 
 $\varphi\in\P_{\m}$ as a tensor in $\mathbf{V}_{b,d,m}$ with $m\neq \bar m$.
An exact representation is possible only if $\m\leq m$.
Otherwise we have to settle for an approximation.
In this section, we consider a particular type of approximation based on local interpolations 
 that we will used in the next section.
\begin{definition}[Local Interpolation]\label{local-interpolation}
We consider an interpolation operator $\mathcal{I}_{m}$ from  $L^p([0,1))$ to $S:=\P_{m}$, $1\le p\le \infty$, such that for all $v\in W^{m+1,p}$
and all $l=0,\ldots,m+1$, it holds
\begin{equation}\label{property-local-interpolation}
    \snorm{v - \mathcal{I}_{m} v}[{\Wkp[l][@]}]\le C \snorm{v}[W^{m+1,p}]
\end{equation}
for some constant $C>0$ independent of $v$.
For the construction of this operator and a proof of the above
property
see, e.g., \cite[Theorem 1.103]{ern2004}.
Then, we introduce the operator $\mathcal{I}_{b,d,m} :=  \mathcal{I}_{b, d,S} $ from  $L^p([0,1))$ to $V_{b,d,m}$ defined by \eqref{local-projection}  with $ \mathcal{I}_S=\mathcal{I}_{m}$.
\end{definition}
\begin{lemma}[Ranks of Interpolants of Polynomials]\label{thm:rankspoly}
    For $\varphi\in\P_{\m}$, $\bar m \in \N$,  the interpolant satisfies
    $\interpol[b,d][m](\varphi)\in\Vbd{m}$ and for $1 \le \nu \le d$, 
    \begin{align}\label{eq:boundpoly}
        r_{\nu,d}(\interpol[b,d][m](\varphi))\leq\min\set{b^\nu,\,(m+1)b^{d-\nu},\,\m+1}.
    \end{align}
\end{lemma}

\begin{proof}
Since $\interpol[b,d][m](\varphi)\in\Vbd{m}$, the  bound $r_{\nu,d}(\varphi)\leq\min\set{b^\nu,\,(m+1)b^{d-\nu}}$ is obtained from  \Cref{rank-bound-VbdS-closed-dilation} (ii). Then from \Cref{local-projection-ranks}, we know that 
$r_{\nu,  d}(\interpol[b,  d][m](\varphi)) \le r_{\nu,  d}( \varphi)  $ for all $1\le \nu \le   d$, and we conclude by using \Cref{ranks-polynomials}.\end{proof}
 From \Cref{thm:rankspoly}, we easily deduce
 \begin{proposition}[Complexity for Encoding Interpolants of Polynomials]\label{prop:encoding-poly-interpolant}
 For a polynomial $\varphi \in\P_{\m}$, $\bar m \in \N$, 
     the different complexities from \eqref{eq:complex} for
    encoding the interpolant $\interpol[b,d][m](\varphi)$ of level $d$ and degree $m\le \bar m$ within $\Vb{m}$
    are bounded as
    \begin{align*}
        \cost_{\mc N}(\interpol[b,  d][m](\varphi))&\leq (\bar m+1) d,\\
        \cost_{\mc S}(\interpol[b,  d][m](\varphi))&\leq \cost_{\mc C}(\interpol[b,  d][m](\varphi))\leq b(\bar m +1)^2 d + b(m+1).
    \end{align*}
\end{proposition}
	
\subsection{Fixed Knot Splines}

Let $b,d\in\N$. We divide $[0,1)$ into $N = b^d$ intervals $[x_k, x_{k+1})$
with
\begin{align*}
    x_k:=kb^{-d},\quad k=0,\ldots,b^d.
\end{align*}
Fix a polynomial of degree $m\in\N_0$ and
a continuity index $\cont\in\N_0\cup\set{-1,\infty}$.
Define the space of fixed knot splines of degree $m$ 
with $N+1$ knots and $\cont$ continuous derivatives as
\begin{align*}
    \S{N}{m}{\cont}:=\set{f:[0,1)\rightarrow\R \;: \; f_{|_{[x_k,x_{k+1})}}\in\P_m,\;k=0,\ldots,N-1\text{ and }
    f\in\C[\cont]{([0,1))}},
\end{align*}
where $\C[-1]{([0,1))}$ stands for not necessarily continuous functions on $[0,1)$,
$\C[0]{([0,1))}$ is the space $\C{([0,1))}$ of continuous functions on $[0,1)$
and $\C[k]{([0,1))}$, $k\in\N\cup\set{\infty}$ is the usual space of
$k$-times differentiable functions.
The following property is apparent.

\begin{lemma}[Dimension of Spline Space]\label{lemma:dimsplinespace}
      $\S{N}{m}{\cont}$ is a finite-dimensional subspace of $\Lp$ with
      \begin{align*}
          \dim\S{N}{m}{\cont}=
          \begin{cases}
              (m+1)N-(N-1)(\cont+1),&\quad -1\leq \cont\leq m,\\
              m+1,&\quad m+1\leq \cont\leq\infty. 
          \end{cases}
      \end{align*}
\end{lemma}

With the above Lemma we immediately obtain

\begin{lemma}[Ranks of Fixed Knot Splines]\label{thm:fixedknot}
    Let $\varphi\in\S{N}{m}{\cont}$.
    Then, $\varphi\in\Vbd{m}$ and for $1\leq \nu\leq d$
    \begin{align}\label{eq:ranksfixedsplines}
        r_{\nu,d}(\varphi)\leq
        \begin{cases}
                \min\set{(m-\cont)b^{d-\nu}+(\cont+1),\;b^\nu},&\quad
                -1\leq \cont\leq m,\\
                \min\set{m+1,\;b^\nu},&\quad m+1\leq\cont\leq\infty.
        \end{cases}
    \end{align}
\end{lemma}
\begin{proof}
For any $0\le j < b^\nu$, the restriction of $\varphi $ to the interval $[b^{-\nu}j,b^{-\nu}(j+1))$ is a piece-wise polynomial in $\C[\cont]{([b^{-\nu}j,b^{-\nu}(j+1)))}$ with $b^{d-\nu}$ pieces, so that  $\varphi(b^{-\nu}(j+\cdot)) \in \S{b^{d-\nu}}{m}{\cont}$ (with knots $ kb^{-\nu}$, $0\le k<b^{\nu} $).  \Cref{link-ranks-d-nu} then implies 
$r_{\nu,d}(f) \le \dim(\S{b^{d-j}}{m}{\cont}) $ and  we obtain
    \eqref{eq:ranksfixedsplines} by using \Cref{lemma:dimsplinespace} and \Cref{rank-bound-VbdS-closed-dilation}.
\end{proof}

\begin{remark}[General Tensor Formats]\label{remark:htsplines}
    We could generalize the above statement to a general tree-based tensor format.
    In this case, for $\beta\subset\set{1,\ldots,d}$ we would have the bound (see
    also \cite[Lemma 2.26]{partI})
    \begin{align*}
         r_{\beta,d}(\varphi)\leq\min\set{(m+1)b^{d-\#\beta},\;b^{\#\beta}}.
    \end{align*}
    Note that $(T_{b,d} \varphi)(j_\beta,\cdot)$ is not necessarily
    a contiguous piece of $\varphi$, 
    even if $\beta$ is a contiguous subset of $\set{1,\ldots,d}$,
    e.g., $\beta=\set{j,j+1,\ldots,j+i}$.
    Therefore additional continuity constraints on
    $\varphi\in\S{N}{m}{\cont}$
    would in general not affect the rank bound.
    Of course, for large $d$ the rank reduction due to continuity constraints is not essential,
    unless $\cont=m$ and in this case the ranks would be
    bounded by $m+1$ in any format.
    See also \cite[Remark 3.3]{partI}.
\end{remark}
\begin{proposition}[Complexity for Encoding Fixed Knot Splines]\label{cor:fixedknot}
    For a fixed knot spline $\varphi\in\S{N}{m}{\cont}$ with $N = b^d$,
    the different complexities from \eqref{eq:complex} for 
    encoding within $\Vb{m}$
    are bounded as
    \begin{align*}
        \cost_{\mc N}(\varphi)&\leq C\sqrt{N},\\
        \cost_{\mc S}(\varphi)&\leq\cost_{\mc C}(\varphi)\leq CN,
    \end{align*}
    with constants $C>0$ depending only on $b$ and $m$.
\end{proposition} 	
\begin{proof}
From \Cref{thm:fixedknot}, we obtain 
\begin{align*}
\cost_{\mathcal{N}}(\varphi) &= \sum_{\nu=1}^{\lfloor d/2 \rfloor} b^\nu +   \sum_{\nu=\lfloor d/2 \rfloor+1}^{d} b^{d-\nu}\le 2\frac{b}{b-1} b^{d/2} =  2\frac{b}{b-1} \sqrt{N},\\
\cost_{\mathcal{C}}(\varphi) &=\sum_{\nu=1}^{\lfloor d/2 \rfloor} b^{2\nu} +   \sum_{\nu=\lfloor d/2 \rfloor+1}^{d} b^{2(d-\nu+1)} + b(m+1) \\
&\le \frac{2b^2}{b^2-1} b^{d} +b(m+1)  = \max\{\frac{2b^2}{b^2-1},(m+1)\} N .
\end{align*}
and we conclude by noting that $\cost_{\mc S}(\varphi)\leq\cost_{\mc C}(\varphi)$.
\end{proof}
	
Now we would like to encode splines of degree $\bar m$ in $V_{b,\bar d,m}$ with $m \neq \bar m$ and $\bar d \ge d$. 
An exact representation is not possible for $\bar m>m$. Then, we again consider the local interpolation operator from \Cref{local-interpolation}. 
\begin{lemma}[Ranks of Interpolants of Fixed Knot Splines]\label{lem:boundrankssplines}
    Let $\varphi\in\S{N}{\bar m}{\cont}$ with $N = b^d$. For $\bar d \ge d$,  
    the interpolant 
    $\interpol[b,\bar d][m](\varphi) \in V_{b,\bar d , m}$ satisfies  
      \begin{align*}
       & r_{\nu,\bar d}(\interpol[b,\bar d][m](\varphi))  \leq
        \begin{cases}
                \min\set{(\bar m-\cont)b^{d-\nu}+(\cont+1),\;b^\nu},&\quad
                -1\leq \cont\leq \bar m,\\
                \min\set{\bar m+1,\;b^\nu},&\quad \bar m+1\leq\cont\leq\infty.
        \end{cases}
       ,\quad  1\le \nu \le d,\\
      &  r_{\nu,\bar d}(\interpol[b,\bar d][m](\varphi))   \leq \min\set{(m+1)b^{\bar d-\nu},\,\m+1}, \quad d < \nu \le \bar d.
    \end{align*}
\end{lemma}
\begin{proof}
From \Cref{local-projection-ranks}, we know that $r_{\nu,\bar d}(\interpol[b,\bar d][m](\varphi)) \le r_{\nu,\bar d}( \varphi)  $ for all $1\le \nu \le \bar d$. For $\nu\le d$, we have  from \Cref{link-ranks-d-nu} that $r_{\nu,\bar d}( \varphi) = r_{\nu,d}(\varphi)$. Then, we obtain the first inequality from  \Cref{thm:fixedknot}. Now consider the case $d < \nu \le \bar d$. 
The bound $ r_{\nu,\bar d}(\interpol[b,\bar d][m](\varphi)) \le (m+1)b^{\bar d-\nu}  $ simply follows from the fact that $\interpol[b,\bar d][m](\varphi) \in V_{b,\bar d,m}$. Since $\varphi \in V_{b,d,\bar m}$ and $\P_{\bar m}$ is closed under dilation, we obtain from 
  \Cref{rank-bound-VbdS-closed-dilation} the other bound $r_{\nu,\bar d}( \varphi) \le \bar m+1.$
\end{proof}
\begin{proposition}[Complexity for Encoding Interpolants of Fixed Knot Splines]\label{cor:fixedknotinterpolant}
    For a fixed knot spline $\varphi\in\S{N}{\m}{\cont}$ with $N = b^d$,
    the different complexities from \eqref{eq:complex} for
    encoding the interpolant $\interpol[b,\bar d][m](\varphi)$  of level $\bar d \ge d$ and degree $m\le \bar m$ within $\Vb{m}$
    are bounded as
    \begin{align*}
        \cost_{\mc N}(\interpol[b,\bar d][m](\varphi))&\leq C\sqrt{N} + C'(\bar d-d),\\
        \cost_{\mc S}(\interpol[b,\bar d][m](\varphi))&\leq\cost_{\mc C}(\interpol[b,\bar d][m](\varphi))\leq CN + C' (\bar d-d),
    \end{align*}
    with constants $C,C'>0$ depending only on $b$, $m$ and $\bar m$.
\end{proposition} 	
\begin{proof}
Using \Cref{local-projection-ranks} and \Cref{thm:fixedknot} and following the proof of \Cref{cor:fixedknot}, we have 
\begin{align*}
\cost_{\mathcal{N}}(\interpol[b,\bar d][m](\varphi)) &\le \sum_{\nu=1}^d r_\nu(\varphi) + (\bar d-d) (\bar m+1) \le 
\frac{2b}{b-1} \sqrt{N} + (\bar d-d) (\bar m+1),\\
\cost_{\mathcal{C}}(\interpol[b,\bar d][m](\varphi)) &\le b r_1(\varphi) + \sum_{\nu=1}^d b r_{\nu-1}(\varphi)r_{\nu}(\varphi) + (\bar d-d) b(\bar m+1)^2 + b (m+1) \\
&\le  \max\{\frac{2b^2}{b^2-1},(m+1)\} N + (\bar d-d) b (\bar m+1)^2,
\end{align*}
and we conclude by noting that $ \cost_{\mc S}(\interpol[b,\bar d][m](\varphi))\leq\cost_{\mc C}(\interpol[b,\bar d][m](\varphi))$.
\end{proof}


\subsection{Free Knot Splines}

A free knot spline is a piece-wise polynomial function,
for which only the maximum polynomial order and the number of polynomial pieces is known
-- not the location of said pieces.
More precisely, the set of free knot splines of degree  $m\in\N_0$ with $N\in\N$ pieces is defined as
\begin{align*}
    \Sfree{N}{m}:=\set{f:[0,1)\rightarrow\R:\;\exists\xk\subset[0,1]\text{ s.t.\ }
    0=x_0<x_1<\ldots<x_N=1\text{ and } f_{|_{[x_k,x_{k+1})}}\in\P_m}.
\end{align*}

Clearly $\Sfree{N}{m}$ is not a linear subspace like  $\S{N}{m}{\mathfrak{c}}$.
Rank bounds for free knot splines are slightly more tricky than for fixed knot splines.
We proceed in three steps:
\begin{enumerate}
    \item\label{point1}    Assume first the knots $x_k$ of the free knot spline are all located on a multiple of $b^{-d_k}$ for some $d_k\in\N$,
                 i.e., only $b$-adic knots are allowed.
                 Assume also the largest $d_k$ is known.
    \item\label{point2}    Show that restricting to $b$-adic knots does not affect the approximation class as compared to
                 non-constrained free knot splines.
    \item\label{point3}    Show that the largest $d_k$ can be estimated using the desired approximation accuracy and
                 excess regularity/integrability of the target function.
\end{enumerate}
In this section, we only address point (\ref{point1}). In \Cref{sec:subbesov}, we will address (\ref{point2}) and (\ref{point3}).

\begin{definition}[Free $b$-adic Knot Splines]
    We call a sequence of points $\xkb\subset[0,1]$ \emph{$b$-adic} if
    \begin{align*}
        x_k^b=i_kb^{-d_k},
    \end{align*}
    for some $d_k\in\N$ and $0\leq i_k\leq b^{d_k}$.
    We use the superscript $b$ to indicate that a sequence is $b$-adic. With this we define the set of \emph{free $b$-adic knot splines} as
    \begin{align*}
            \Sfreeb{N}{m}:=\set{f:[0,1)\rightarrow\R:\;\exists\xkb\subset[0,1]\text{ s.t.\ }
            0=x^b_0<x^b_1<\ldots<x^b_N=1\text{ and } f_{|_{[x^b_k,x^b_{k+1})}}\in\P_m}.
    \end{align*}
\end{definition}

\begin{lemma}[Ranks of Free $b$-adic Knot Splines]\label{thm:freeknot}
    Let $\varphi\in\Sfreeb{N}{m}$ with $\xkb$ being the $b$-adic knot sequence
    corresponding to $\varphi$. Let $d:=\max\set{d_k:\;1\leq k\leq N-1}$.
    Then, $\varphi\in\Vbd{m}$ and
    \begin{align}\label{eq:boundfreeknot}
        r_{\nu,d}(\varphi)\leq\min\set{b^\nu,\,(m+1)b^{d-\nu},\,m+N}
    \end{align}
    for $1\le \nu\le d$. 
\end{lemma}
\begin{proof}
\label{proof:freeknot}
For any $0\le j < b^\nu$, the restriction of $\varphi$ to the interval $[b^{-\nu}j,b^{-\nu}(j+1))$ is  either a polynomial or
    a piece-wise polynomial where
    the number of such piece-wise polynomials is at most $N-1$,
    since there are at most $N-1$ discontinuities in $(0,1)$.
    Hence, \Cref{link-ranks-d-nu} implies that 
   $
        r_\nu(\varphi)\leq m+N 
  $ 
  for all $1\leq \nu \leq d$,  and we obtain the other bound $r_{\nu,d}(\varphi)\leq\min\set{b^\nu,\,(m+1)b^{d-\nu}}$ from
\Cref{rank-bound-VbdS-closed-dilation}  with $\dim(S) = m+1$.
\end{proof}

\begin{proposition}[Complexity for Encoding Free Knot Splines]\label{cor:freeknot}
    For a free knot spline $\varphi\in\Sfreeb{N}{m}$
    with $d:=\max\set{d_k:\;1\leq k\leq N-1}$,
    the different complexities from \eqref{eq:complex} for
    encoding within $\Vb{m}$
    are bounded as
    \begin{align*}
        \cost_{\mc N}(\varphi)&\leq CdN,\\
        \cost_{\mc C}(\varphi)&\leq CdN^2,\\
        \cost_{\mc S}(\varphi)&\leq Cd^2N,
    \end{align*}
    with constants $C>0$ depending only on $b$ and $m$.
\end{proposition} 
\begin{proof}
Follows from \Cref{thm:freeknot}, cf. also \Cref{cor:fixedknot}.
\end{proof}

\begin{lemma}[Ranks of Interpolants of Free $b$-adic Knot Splines]\label{lem:boundranksfreeknotsplines}
    Let $\varphi\in\Sfreeb{N}{\bar m}$, $\bar m\ge m$, and $\xkb$ being the $b$-adic knot sequence
    corresponding to $\varphi$. Let $d:=\max\set{d_k:\;1\leq k\leq N-1}$. For $\bar d \ge d$,  
    the interpolant 
    $\interpol[b,\bar d][m](\varphi) \in V_{b,\bar d , m}$ satisfies  
      \begin{align*}
       & r_{\nu,\bar d}(\interpol[b,\bar d][m](\varphi)) \leq 
        \min\set{b^\nu,\,(\bar m+1)b^{d-\nu},\,\bar m+N} ,\quad  1\le \nu \le d,\\
      &  r_{\nu,\bar d}(\interpol[b,\bar d][m](\varphi))   \leq \min\set{(m+1)b^{\bar d-\nu},\, \bar m +1}, \quad d < \nu \le \bar d.
    \end{align*}
\end{lemma}
\begin{proof}
\label{proof:freeknot2}
From \Cref{local-projection-ranks}, we know that $r_{\nu,\bar d}(\interpol[b,\bar d][m](\varphi)) \le r_{\nu,\bar d}( \varphi)  $ for all $1\le \nu \le \bar d$. For $\nu\le d$, we have  from \Cref{link-ranks-d-nu} that $r_{\nu,\bar d}( \varphi) = r_{\nu,d}(\varphi)$. Then, we obtain the first inequality from  \Cref{thm:freeknot}.
Now consider the case $d < \nu \le \bar d$. 
The bound $ r_{\nu,\bar d}(\interpol[b,\bar d][m](\varphi)) \le (m+1)b^{\bar d-\nu}  $ simply follows from the fact that $\interpol[b,\bar d][m](\varphi) \in V_{b,\bar d,m}$. Since $\varphi \in V_{b,d,\bar m}$ and $\P_{\bar m}$ is closed under dilation, we obtain from 
  \Cref{rank-bound-VbdS-closed-dilation} the other bound $r_{\nu,\bar d}( \varphi) \le  \bar m+1.$
 \end{proof}
\begin{proposition}[Complexity for Encoding Interpolants of Free Knot Splines]\label{cor:freeknotinterpolant}
    For a free knot spline $\varphi\in\Sfreeb{N}{\m}$
    with $d:=\max\set{d_k:\;1\leq k\leq N-1}$,
    the different complexities from \eqref{eq:complex} for
    encoding the interpolant $\interpol[b,\bar d][m](\varphi)$ of level $\bar d \ge d$ and degree $m \le \bar m$
    are bounded as
    \begin{align*}
        \cost_{\mc N}(\interpol[b,\bar d][m](\varphi))&\leq C d N + C' (\bar d - d) ,\\
        \cost_{\mc C}(\interpol[b,\bar d][m](\varphi))&\leq CdN^2 + C' (\bar d - d),\\
        \cost_{\mc S}(\interpol[b,\bar d][m](\varphi))&\leq Cd^2N + C' (\bar d - d),
    \end{align*}
    with constants $C,C'>0$ depending only on $b$, $m$ and $\bar m$.
\end{proposition} 
\begin{proof}
Follows from \Cref{lem:boundranksfreeknotsplines}, cf. also
\Cref{cor:fixedknotinterpolant}.
\end{proof}

\begin{remark}\label{rem:simplebound}
    Both in \Cref{thm:freeknot} and \Cref{lem:boundranksfreeknotsplines}, the rank
    bound is of the order $N$
    and does
    not assume any specific structure of the spline approximation.
    This is a crude estimate that could be perhaps improved
    if one imposes additional restrictions, such as a tree-like support
    structure of the approximating splines.
    See also \Cref{rem:waveletssimple}.
\end{remark}



\subsection{Multi-Resolution Analysis}

We have already mentioned a connection between $\Vb{S}$ and MRAs: see
Definition \ref{def:badicdil}.
In this subsection, we further explore if and how MRAs are
intrinsically encoded within $\Vb{S}$.
Specifically, we consider the following three scenarios:
\begin{itemize}
    \item    $S=\P_m$ and piecewise polynomial MRAs.
    \item    $S$ itself contains the generators of the MRA.
    \item    We can \emph{approximate} the generators of the MRA
                with functions in $S$ or $V_{b,d,S}$ upto a fixed accuracy $\varepsilon>0$.
\end{itemize}

For a review of MRAs we refer to, e.g., \cite[Chapter 7]{Mallat2009}. 
An MRA of the space $\Lp$ ($1\leq p<\infty$)
consists of a sequence of spaces
\begin{align*}
    V_0\subset V_1\subset V_2\subset\ldots\subset \Lp.
\end{align*}
Usually one considers MRAs on $\Lp{(\R)}$,
but they can be adapted to bounded domains, see \cite{Cohen1997}.
The sequence $V_j$ is required to satisfy
certain properties, such as
invariance under $b$-adic dilation and shifting
-- the counterpart of Definition \ref{def:badicdil}.
Another property is that the sequence $V_j$ is generated by
dilating and shifting one or a finite number of so-called
\emph{generating} functions.

For a function $\psi:[0,1)\rightarrow\R$, we define
the level $l\in\N_0$ $b$-adic dilation, shifted by $j=0,\ldots,b^l-1$, as
\begin{align*}
    \li{\psi}(x):=   \begin{cases} b^{l/p}\psi(b^lx-j) & \text{for} \quad x\in [ b^{-l}j , b^{-l}(j+1))  \\
    0 & \text{elsewhere}.
    \end{cases}
\end{align*}
The factor $b^{l/p}$ normalizes these functions in $\Lp$.
The purpose of this subsection is to illustrate
the interplay between the
spaces generated by such functions $\li{\psi}$ and $\Vb{S}$.
The following result provides the tensorized representations of  functions $\li{\psi}$.

\begin{lemma}[$b$-adic dilations]\label{lemma:badicdils}
    Let $l \in \N$ and $j=\sum_{k=1}^l j_kb^{l-k}$. Then 
    \begin{align*}
       T_{b,l} \li{\psi} = e_{j_1}^p\otimes\cdots\otimes e_{j_l}^p\otimes \psi 
    \end{align*}
    with $e_{j_k}^p=b^{1/p}\delta_{j_k},$ and for all $d\ge l$,
    \begin{align*}
       T_{b,d} \li{\psi} = e_{j_1}^p\otimes\cdots\otimes e_{j_l}^p\otimes (T_{b,d-l} \psi).
    \end{align*}
\end{lemma}
\begin{proof}
The expression of $T_{b,l}  \li{\psi}$ follows from
\cite[Corollary 2.5]{partI}.
The second property follows from
\cite[Lemma 2.6]{partI}. 
\end{proof}
 
 \Cref{lemma:badicdils} immediately gives

\begin{corollary}\label{cor:dilation}
    For any subspace $S\subset\Lp$, $0\leq l,d_0$
    and $0\leq j\leq b^l-1$, we have
    \begin{align*}
        \li{\psi}\in\Vbd[@][d_0+l]{S}\quad\Leftrightarrow\quad
        \psi\in\Vbd[@][d_0]{S}.
    \end{align*}
    Moreover, if $\psi\in\Vbd[@][d_0]{S}$ satisfies
    \begin{align*}
        r_{\nu,d_0}(\psi)\leq r_{\nu,d_0},\quad 1\leq \nu\leq d_0,
    \end{align*}
    then for $ {d} \ge d_0$ and $l:= {d}-d_0$,
    we have $\li{\psi}\in\Vbd[@][d]{S}$ and
    \begin{align*}
        r_{\nu,d}(\li{\psi})=1\text{ for }1\leq \nu\leq l,\quad\text{and}\quad
        r_{\nu,d}(\li{\psi})\leq r_{\nu-l,d_0}\text{ for }l<\nu\leq d.
    \end{align*}
\end{corollary}

\subsubsection{Piecewise Polynomial MRA}

Let $\psi:[0,1)\rightarrow\R$ be a piece-wise polynomial of degree at most $m$
and let the discontinuities be located on a subset of
\begin{align*}
    \set{x_k:=kb^{-d_0}:\;k=0,\ldots,b^{d_0}}.
\end{align*}
Then, clearly $\psi\in V_{b,d_0,m}$. For $1\le \nu \le d_0$, the ranks can be bounded as in
Theorem \ref{thm:fixedknot}
\begin{align*}        
        r_{\nu,d_0}(\psi)\leq\min\set{(m-\cont)b^{d_0-\nu}+(\cont+1),\;b^\nu},
\end{align*}
where $\cont$ is the number of continuous derivatives.
Note that we can generalize this to arbitrary $b$-adic
discontinuity knots and proceed as in Theorem \ref{thm:freeknot},
but we restrict ourselves to the above setting for the
purpose of this presentation. Moreover, if $\psi$ is a typical
wavelet or scaling function, we expect $d_0$ and consequently $r_\nu(\psi)$
to be small.

By Lemma \ref{rank-bound-VbdS-closed-dilation},
for any $d\geq d_0$ we also have $\psi\in\Vbd[@][d]{m}$
with
\begin{align*}
 &r_{\nu,d}(\psi)\leq \min\set{(m-\cont)b^{d_0-\nu}+(\cont+1),\;b^\nu},\quad 1 \le \nu \le d_0,
\\  &  r_{\nu,d}(\psi)\leq m+1,\quad  d_0 \le \nu\leq d.
\end{align*}

In summary, if $\psi$ is piece-wise polynomial as above, then, since
$S=\P_m$ is closed under $b$-adic dilation, we know that
\begin{enumerate}[label=\alph*)]
    \item    $\psi\in\Vbd[@][d]{m}$ for any $d\geq d_0$ with ranks
                bounded by \Cref{rank-bound-VbdS-closed-dilation} and 
                \Cref{thm:fixedknot},
    \item    and $\li{\psi}\in\Vbd[@][{d}]{m}$ for any
                ${d}\geq d_0+l$ by
                Corollary \ref{cor:dilation} with similar rank bounds.
\end{enumerate}
 \begin{example}[Haar functions]\label{ex:haar}
 The Haar mother wavelet $\psi(x) = - \indicator{(0,1/2)}(x) + \indicator{(1/2,1)}(x)$ is such that $\psi \in V_{2,1,0}$ ($b=2$, $d_0 = 1$, $m=0$). Its tensorization at level $1$ is  $T_{2,1} \psi(i_1,y)= -\delta_0(i_1) + \delta_1(i_1) :=h(i_1)$ and its rank $r_{1,1}(\psi) = 1.$ Since $T_{2,1} \psi$ does not depend on $y$, we have
 $T_{2,d} \psi(i_1,\hdots,i_d,y) = h(i_1)$ and $r_{\nu,d}(\psi)=1$ for all $d\in \N$ and $1\le \nu \le d$. The Haar wavelets $ \psi_{2,l,j} = 2^{-l/2} \psi(2^lx-j)$ are such that for all $d \ge l+1,$ $T_{b,d} \psi_{2,l,j}(i_1,\hdots,i_d,y) = e^2_{j_1}(i_1)\hdots e^2_{j_{l-1}}(i_{l-1}) h(i_l)  $, and $r_{\nu,d}(\psi_{2,l,j})=1$ for all $1\le \nu\le d.$
 \end{example}
  \begin{example}[Hat functions]\label{ex:hat}
The hat function $\psi(x) = 2x \indicator{(0,1/2)}(x) + 2(1-x) \indicator{(1/2,1)}(x)$ is such that $\psi \in V_{2,1,1}$ ($b=2$, $d_0 = 1$, $m=1$). Its tensorization at level $1$ is  $T_{2,1} \psi(i_1,y)= \delta_1(i_1) + (\delta_0(i_1) - \delta_1(i_1)) y $ and its rank $r_{1,1}(\psi) = 2.$ For any $d \ge 1$, $T_{2,d} \psi(i_1,\hdots,i_d,y) = \delta_1(i_1) + (\delta_0(i_1) - \delta_1(i_1)) t_{2,d-1}(i_2,\hdots,i_d,y)$, with $t_{2,d-1}(i_2,\hdots,i_d,y )= 2^{-d+1} ( \sum_{k=1}^{d-1} i_{k+1} 2^{k} + y)$, an expression from which we deduce that  
  $r_{\beta,d}(\psi) \le 2$ for all $\beta \in \{1,\hdots,d+1\}$. The functions $ \psi_{2,l,j} = 2^{-l/2} \psi(2^lx-j)$ are such that for all $d \ge l+1,$ $T_{b,d} \psi_{2,l,j}(i_1,\hdots,i_d,y) = e^2_{j_1}(i_1)\hdots e^2_{j_{l}}(i_{l}) T_{b,d-l} \psi(i_{l+1},\hdots,i_d,y) $, and $r_{\beta,d}(\psi_{2,l,j})\le 2$ for all $\beta \in \{1,\hdots,d+1\}.$
 \end{example}

\subsubsection{$S$ Contains MRA Generators}

Suppose $\psi$ from above is a mother wavelet and $\psi\in S$.
Then, by Corollary \ref{cor:dilation}, $\li{\psi}\in\Vbd[@][l]{S}$.
However, we do not necessarily have $\li{\psi}\in\Vbd[@][d]{S}$
for $d>l$. This presents a problem when considering
two wavelets on different levels, since the sum
may not belong to $\Vbd{S}$ for any $d$
(see also \cite[Example 2.20]{partI}).
Thus, as before we require $S$ to be closed under
$b$-adic dilation.
It is not difficult to see that this is the case exactly
when $S$ includes
\emph{all} generators for the MRA, i.e., all mother scaling functions
and their shifts.

Put more precisely, let $\varphi^q$, $1\le q \le Q$, be mother scaling functions,
and $\psi^q$ the corresponding mother wavelets.
We intentionally include the possibility of multiple
scaling functions and wavelets since this is the relevant setting
for wavelets on bounded domains or orthogonal multi-wavelets.
The shift-invariant setting on unbounded domains
with a single scaling function and infinitely many integer shifts
can be handled similarly.

We assume these functions satisfy refinement relationships of the form
\begin{align}\label{eq:refinement}
    \varphi^q &= \sum_{q,i} a^q_i\varphi^q(b\cdot-i),\notag\\
    \psi^q &= \sum_{q,i} c^q_i\varphi^q(b\cdot-i),
\end{align}
where the number of non-zero $a^q_i\neq 0$, $c^q_i\neq 0$ is typically
smaller than the total number of scaled and
shifted $\varphi^q(b\cdot-i)$.
The refinement property
is standard for MRAs.
 \begin{example}[Haar functions]
Consider again \Cref{ex:haar}. Here $Q=1$, with mother Haar scaling function $\varphi(x) = \indicator{[0,1)}(x)$ such that $\varphi(x) = \varphi(2 x) + \varphi(2x-1)$, and the mother Haar wavelet $\psi(x) = - \varphi(2 x) + \varphi(2x-1)$.
\end{example}	

\begin{proposition}[Ranks for MRA]
    Suppose we have
    \begin{align*}
        S=\linspan{\varphi^q:\;q=1,\ldots,Q},
    \end{align*}
    with functions satisfying \eqref{eq:refinement}.
    Then, $S$ is closed $b$-adic dilation.
    Moreover, by the refinement relation,
    $\psi^q\in\Vbd[@][1]{S}$ and therefore for the ranks we obtain
    \begin{alignat*}{3}
        \li{\varphi^q}&\in\Vbd[@][d]{S},\; l\geq 0,\;d\geq l,\quad
        &&r_\nu(\li{\varphi^q})&&=1,\;1\leq \nu\leq l,\\
        & &&r_\nu(\li{\varphi^q})&&\leq\dim S,\;l< \nu\leq d,\\       
        \li{\psi^q}&\in\Vbd[@][d]{S},\;l\geq 0,\;d\geq l+1,\quad
       &&r_\nu(\li{\psi^q})&&=1,\;1\leq \nu\leq l,\\
       & &&r_\nu(\li{\psi^q})&&\leq \dim S,\;l<\nu\leq d.
    \end{alignat*}
\end{proposition}

\begin{proof}
    An application of Corollary \ref{cor:dilation} and \eqref{eq:refinement}.
\end{proof}

\subsubsection{Approximate MRA Representations}

Let $\eta>0$ and suppose that for some $d=d(\eta)$ and ranks $r_{\nu}(\eta)$,
$1\leq \nu\leq d(\eta)$, there exists a $v^\eta\in\Vbd{S}$ such that
\begin{align*}
    \norm{\psi-v^\eta}[p]\leq\eta\quad r_{\nu}(v^\eta)\leq r_{\nu}(\eta),
\end{align*}
where $\psi$ is a mother wavelet.
In the following results $\psi$ can be replaced by
a scaling function $\varphi$ without any changes.

For $l\geq 0$, $0\leq j\leq b^l-1$ and $j=\sum_{k=1}^l j_k b^{l-k}$, we define
\begin{align}\label{vetaj-tensorized}
    v^\eta_{b,l,j}:= T_{b,l}^{-1} (e_{j_1}^p \otimes\cdots\otimes e_{j_l}^p \otimes v^\eta), 
\end{align}
where $e_{j_k}^p$ are defined as in Lemma \ref{lemma:badicdils}.
The ranks of $v^\eta_{b,l,j}$ are given by Corollary \ref{cor:dilation}.
Moreover, $v^\eta_{b,l,j}$ approximates $\li{\psi}$ at least as well
as $v^\eta$ approximates $\psi$.

\begin{lemma}[Approximate Tensorized MRA]\label{lemma:appmra}
    Let $\psi$, $v^\eta$ be given as above and $1\leq p<\infty$. Then,
    \begin{align*}
        \Vert{\li{\psi}-v^\eta_{b,l,j}}\Vert_p \leq\eta.
    \end{align*}
\end{lemma}

\begin{proof}
From \eqref{vetaj-tensorized} and \Cref{lemma:badicdils}, we have $T_{b,d}(\li{\psi}-v^\eta_{b,l,j}) = e_{j_1}^p \otimes\cdots\otimes e_{j_l}^p \otimes (\psi - v^\eta).$ Then.
\Cref{thm:tensorizationmap} and the crossnorm property (see also \cite{partI})
imply $$  \Vert{\li{\psi}-v^\eta_{b,l,j}} \Vert_p = \Vert e_{j_1}^p \Vert_{\ell^p} \hdots \Vert e_{j_l}^p \Vert_{\ell^p} \Vert \psi - v^\eta \Vert_p = \Vert \psi - v^\eta \Vert_p \le \eta.$$
\end{proof}

Let $f\in\Lp$ be an arbitrary function, $\varepsilon>0$ and $f_N$ an
$N$-term wavelet approximation with $N=N(\varepsilon)$.
I.e.,
\begin{align*}
    \norm{f-f_N}[p]\leq\varepsilon,\quad f_N=
    \sum_{\lambda\in\Lambda}c_\lambda\psi_\lambda,
\end{align*}
where $\Lambda\subset\mc J$ is a
generalized index set\footnote{E.g., each index is of the form
 $\lambda=(b,l,j)$.} with $\#\Lambda\leq N$. We assume that $\psi_\lambda$ is normalized in $\Lp$.
Suppose each $\psi_\lambda$ can be approximated via
$v^\eta_\lambda\in\Vbd{S}$ as in Lemma \ref{lemma:appmra}.
This gives an approximation $v^\eta\in\Vbd{S}$ to $f$ via
\begin{align*}
    v^\eta=\sum_{\lambda\in\Lambda}c_\lambda v^\eta_\lambda.
\end{align*}
By Lemma \ref{lemma:appmra}, we can bound the error as
\begin{align}
    \norm{f_N-v^\eta}[p]\leq\sum_{\lambda\in\Lambda}|c_\lambda|
    \norm{\psi_\lambda-v^\eta_\lambda}[p]
    \leq\eta\sum_{\lambda\in\Lambda}|c_\lambda|.
    \label{apptruncatedwavelet}
\end{align}
A wavelet basis allows to characterize
different norms of $f$
by its coefficients in the basis,
see \cite{DeVore1997} for unbounded domains and
\cite{Cohen1997} for bounded domains.

In the following we assume that the wavelet basis
$\Psi:=\set{\psi_\lambda:\;\lambda\in\mc J}$ is a normalized basis of $L^p$.
A function $f \in \Lp$ admits a representation 
$$
f = \sum_{\lambda \in \mc J} c_\lambda \psi_\lambda.
$$
Denoting by $|\lambda|$ the dilation level of $\psi_\lambda$, we can introduce an equivalent representation
\begin{align}
f = \sum_{\lambda \in \mc J} c_{\lambda,p'} \psi_{\lambda,p'},\quad c_{\lambda,p'} = b^{-\vert \lambda \vert(\frac{1}{p'} - \frac{1}{p})}c_\lambda , \quad \psi_{\lambda,p'}= b^{\vert \lambda \vert(\frac{1}{p'} - \frac{1}{p})} \psi_\lambda ,\label{changeLpLp'}
\end{align}
where the functions $\psi_{\lambda,p'}$ are normalized in $L^{p'}.$
Then, we assume that the wavelet basis 
 has the property that
for any function $f\in B^{\alpha}_{p',q}$ (see \Cref{def:besov} of Besov spaces), $0<p',q\leq\infty$, $\alpha>0$, it holds
\begin{align}\label{eq:besovchar}
        \snorm{f}[B^{\alpha}_{p',q}]^q\sim\sum_{l=0}^\infty b^{l\alpha q}
        \left(\sum_{\lambda\in\mc J_l} | c_{\lambda,p'}|^{p'}\right)^{q/{p'}}
    \end{align}    
     where
      $\mc J_l$ is the subset of $\mc J$ consisting of only
    level $|\lambda|=l$ indices.
    The condition \eqref{eq:besovchar} is satisfied for wavelets with sufficient
regularity and vanishing moments,
see \cite{DeVore98, Cohen1997, DeVore1997}.
In fact, we do not need \eqref{eq:besovchar} to hold for
any $0<p,q\leq\infty$ as the next proposition shows.

\begin{proposition}[Approximate $N$-Term Expansion]
    Let
    \begin{align*}
        f=\sum_{\lambda\in\mc J}c_{\lambda}\psi_\lambda
    \end{align*}
    be the wavelet expansion of $f\in\Lp$ for $1<p<\infty$,
    where the $\psi_\lambda$ are normalized in $\Lp$.
    Let $v^\eta$ be as above,
    assume $f\in\Bsqp[@][1][1]$ for
    $0<\alpha<1$ satisfying
    \begin{align}\label{eq:pstar}
        \alpha+1/p=1,
    \end{align}
    and assume the characterization \eqref{eq:besovchar}
    holds for $\Bsqp[@][1][1]$.
    
    Then,
    \begin{align*}
        \norm{f_N-v^\eta}[p]\leq C\snorm{f}[{\Bsqp[@][1][1]}]\eta,
    \end{align*}
    where $C=C(\Psi)>0$ is a constant independent of $N$,
    $\varepsilon$, $\eta$ or $f$.
\end{proposition}

\begin{proof}
 Using  \eqref{eq:besovchar} with $p'=q=1$, \eqref{eq:pstar} and then \eqref{changeLpLp'} with $p'=1$, we obtain
   $$
   \vert f \vert_{B^\alpha_{1,1}} \sim \sum_{\lambda} b^{\vert \lambda \vert \alpha} \vert c_{\lambda,1} \vert = 
     \sum_{\lambda}  b^{\vert \lambda \vert (1-\frac{1}{p})}  \vert c_{\lambda,1} \vert =  \sum_{\lambda} \vert  c_{\lambda} \vert.
   $$
  Then we conclude by using  \eqref{apptruncatedwavelet}.
    \end{proof}

The condition of excess regularity $f\in B^\alpha_{1,1}$ with $\alpha = 1-1/p$ can be
replaced by excess integrability, at the cost of more complicated
estimates, or excess summability of
the coefficients $c_\lambda$.
The ranks of $v^\eta$ are bounded by $Nr_\nu(\eta)$,
and $d(v)$ is the maximum level among all $d(v^\eta_\lambda)$ that depends
only on $\eta$ and $\varepsilon$. I.e.,
for $\eta\sim\varepsilon$,
we can approximate functions $f\in\Lp$ with the same precision
as the wavelet system
$\Psi:=\set{\psi_\lambda:\;\lambda\in\mc J}$,
where the overall order of the approximation depends
on the approximation order of $\Psi$ and how well
$\Psi$ is approximated by $\Vbd{S}$.

\begin{remark}\label{rem:waveletssimple}
    Similarly to \Cref{rem:simplebound},
    the ranks of $N$-term approximations both for exact MRA and approximate
    MRA representations can be bounded by a multiple of $N$.
    This is a very crude estimate that assumes no specific structure of
    the wavelet approximation.
    Of course, the ranks occurring in practice may be much smaller,
    if we additionally impose restrictions on $\Lambda$,
    e.g., if we require that $\Lambda$ has a tree structure
    (see \cite{Cohen2001}).
\end{remark}
\section{Direct and Inverse Estimates}\label{sec:dirinv}

In this section, we discuss direct and inverse estimates for
the approximation spaces defined in \Cref{sec:appspacestensors}.
Since we verified that $N^\alpha_q$, $C^\alpha_q$
and $S^\alpha_q$ satisfy \ref{P1} -- \ref{P4},
we can use classical approximation theory (see \cite{DeVore93, DeVore98})
to show that an entire scale of interpolation and smoothness spaces
is continuously embedded into these approximation classes.
We begin by briefly reviewing Besov and interpolation spaces.


\subsection{Besov Spaces}

Besov spaces provide a natural framework of smoothness
for approximation theory since the measure of smoothness
in the Besov scale is fine enough to adequately capture
different approximation classes of functions.
At the same time many other smoothness spaces
such as Lipschitz spaces or (fractional) Sobolev spaces
are special cases of Besov spaces.

In principle, there are two standard ways of measuring smoothness:
by considering local rates of change of function values,
or by decomposing a function into certain building blocks and
considering the rates of decay of the high-frequency components.
The original definition of Besov spaces \cite{Besov1959} and the
one we follow here is the former.
The latter is also possible, see \cite{Triebel1992}.

Let $f\in\Lp$, $0<p\leq\infty$ and consider the difference operator
\begin{align*}
    \diff_h&:\Lp{([0,1))}\rightarrow\Lp{([0,1-h))},\\
    \diff_h[f](\cdot)&:=f(\cdot+h)-f(\cdot).
\end{align*}
For $r=2,3,\ldots$, the $r$-th difference is defined as
\begin{align*}
    \diff_h^r:=\diff_h \circ \diff_h^{r-1},
\end{align*}
with $\diff_h^1:=\diff_h$.
\emph{The $r$-th modulus of smoothness} is defined as
\begin{align}\label{eq:modsmooth}
    \mod_r(f,t)_p:=\sup_{0<h\leq t}\norm{\diff_h^r[f]}[p],\quad t>0.
\end{align}

\begin{definition}[Besov Spaces]\label{def:besov}
    For parameters $\alpha>0$ and $0<p,q\leq\infty$, define
    $r:=\lfloor\alpha\rfloor+1$ and the Besov  (quasi-)semi-norm as
    \begin{align*}
        \snorm{f}[\Bsqp]:=
        \begin{cases}
            \left({\int_0^{1}}[t^{-\alpha}\mod_r(f,t)_p]^q\frac{\d t}{t}\right)^{1/q},
            &\quad 0<q<\infty,\\
            {\sup_{0<t\le 1}}t^{-\alpha}\mod_r(f,t)_p,
            &\quad q=\infty.            
        \end{cases}
    \end{align*}
    The Besov (quasi-)norm is defined as
    \begin{align*}
        \norm{f}[\Bsqp]:=\norm{f}[p]+\snorm{f}[\Bsqp].
    \end{align*}
    The \emph{Besov space} is defined as
    \begin{align*}
        \Bsqp:=\set{f\in\Lp:\;\norm{f}[\Bsqp]<\infty}.
    \end{align*}
\end{definition}

In \Cref{def:besov}, any $r$ such that
$r>\alpha$ defines the same space with equivalent norms.
The primary
parameters are $\alpha$ and $p$: $\alpha$
is the order of smoothness, while $p$
indicates the measure of smoothness.
The smaller $p$, the less restrictive the measure of smoothness.
This is particularly important for
free knot spline approximation\footnote{Which can be viewed as the
theoretical idealization of adaptive approximation.},
 where frequently $p<1$,
which allows to measure smoothness for functions
that would be otherwise too irregular in the scale of Sobolev spaces.
The parameter $q$ is secondary,
allowing for a finer gradation of smoothness for the same primary
parameters $\alpha$ and $p$.

Varying the parameters $\alpha$, $p$ and $q$, we have
the obvious inclusions
\begin{align*}
\Bsqp[@][q_1][@]\subset\Bsqp[@][q_2][@],&\quad q_1\leq q_2,\\
\Bsqp[@][@][p_1]\subset\Bsqp[@][@][p_2],&\quad p_1\geq p_2,\\
\Bsqp[\alpha_1][@][@]\subset\Bsqp[\alpha_2][@][@],
&\quad \alpha_1\geq \alpha_2.
\end{align*}
For $\alpha<1$ and $q=\infty$, $\Bsqp[@][\infty][@]=\Lipsp$,
where the latter is the space of functions $f\in\Lp$
such that
\begin{align*}
    \norm{f(\cdot+h)-f}[p]\leq Ch^\alpha.
\end{align*}
If $\alpha$ is an integer and $p\neq 2$,
$\Bsqp[@][\infty][@]$ is slightly larger than the \emph{Sobolev space}
$\Wkp$. For non-integer $\alpha$,
the fractional Sobolev space $\Wkp$ is the same as $\Bsqp[@][p][@]$.
For the special case $p=2$ we even have
$\Wkp[@][2]=\Bsqp[@][2][2]$ for any $\alpha>0$.

To visualize the relationship between the different spaces,
we refer to the DeVore diagram in \Cref{fig:DeVore}.
    \begin{figure}[h]
        \centering
        \includegraphics[scale=.5]{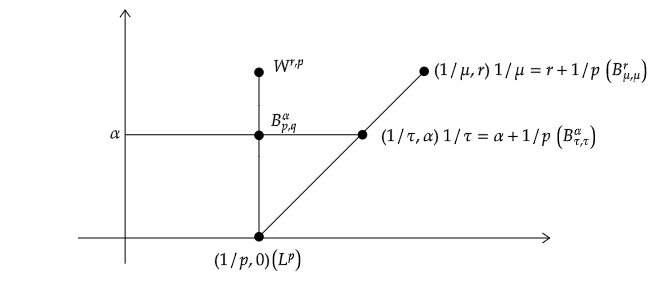}
        \caption{DeVore diagram of smoothness spaces \cite{DeVore98}.
        The Sobolev embedding line is the diagonal
        with the points $(1/\tau,\alpha)$ and $(1/\mu, r)$}.
        \label{fig:DeVore}
    \end{figure}
The $x$-axis corresponds to $1/p$,
where $p$ is the integrability parameter
and primary measure of smoothness.
The $y$-axis corresponds to the smoothness parameter $\alpha$.
A point on this plot represents a space.
E.g., the point $(1/2,0)$ is the space $\Lp[2]$,
$(0,0)$ is $\Lp[\infty]$, $(1/2, 1)$ is $\Wkp[1][2]$ and so on.
Of particular importance is the
\emph{Sobolev embedding line}:
points above this line are embedded in $\Lp$,
points on this line may or may not be embedded into $\Lp$,
and points below this line are never embedded in $\Lp$.
We cite the following important result
for later reference.

\begin{theorem}[{\cite[Chapter 12, Theorem 8.1]{DeVore93}}]
    Let $\alpha>0$, $0<p<\infty$ and define the
    \emph{Sobolev embedding number} $\sobolev$
    \begin{align*}
        \sobolev := \sobolev(\alpha,p) := (\alpha+1/p)^{-1}.
    \end{align*}
    Then,
    \begin{align*}
        \Bsqp[@][\sobolev][\sobolev]\hookrightarrow
        \Bsqp[@][p][\sobolev]\hookrightarrow
        \Lp.
    \end{align*}
\end{theorem}


\subsection{Interpolation Spaces}

Given two linear spaces $X$ and $Y$,
the theory of interpolation spaces allows
to define an entire scale of linear spaces
that are in some sense ``in-between'' $X$ and $Y$.
Moreover, this theory provides a convenient tool
for extending results derived for $X$ and $Y$ to
spaces in-between.

To be more specific, let $X$ and $Y$ be
normed\footnote{Quasi-seminormed spaces would suffice as well. We omit this
for simplicity.} linear
spaces with $Y\hookrightarrow X$.
Let $T$ be any linear operator such that
$T\in\L(X)$ and $T\in\L(Y)$, i.e.,
$T$ maps boundedly $X$  and $Y$ onto itself.
If for any normed linear space $Z$ with $Y\hookrightarrow Z\hookrightarrow X$,
we have $T\in\L(Z)$, then such $Z$ is called
an \emph{interpolation space}.

In this work, we specifically consider one method of constructing
such interpolation spaces: Peetre's $K$-functional.
Let $X$, $Y$ be Banach spaces with $Y\hookrightarrow X$. The $K$-functional
on $X$ is defined as
\begin{align*}
    K(f, t, X, Y):=K(f, t):=\inf_{g\in Y}\set{\norm{f-g}[X]+t\norm{g}[Y]},\quad
    t>0.
\end{align*}

\begin{definition}[Interpolation Spaces, {\cite[Chapter 5]{Bennett}}]
    Define a (quasi-)norm on $X$
    \begin{align*}
        \norm{f}[\theta,q]:=
        \begin{cases}
            \int_0^\infty\left[t^{-\theta}K(f,t)^q\frac{\d t}{t}\right]^{1/q},
            \quad &0<\theta<1,\;0<q<\infty,\\
            \sup_{t>0}t^{-\theta}K(f,t),&0\leq \theta\leq 1,\;\quad q=\infty.
        \end{cases}
    \end{align*}
    The \emph{interpolation space} $(X,Y)_{\theta,q}$ is defined as
    \begin{align*}
        (X,Y)_{\theta,q}:=
        \set{f\in X:\;\norm{f}[\theta,q]<\infty},
    \end{align*}
    and it is a complete (quasi-)normed space.
\end{definition}

Some basic properties of these interpolation spaces are:
\begin{itemize}
    \item $Y\hookrightarrow (X,Y)_{\theta,q}\hookrightarrow X$;
    \item $\XY[\theta_1][@]\hookrightarrow\XY[\theta_2][@]$ for
            $\theta_1\geq\theta_2$ and
            $\XY[@][q_1]\hookrightarrow\XY[@][q_2]$ for
            $q_1\leq q_2$;
    \item \emph{re-iteration} property:
            let $X':=\XY[\theta_1][q_1]$, $Y':=\XY[\theta_2][q_2]$. Then,
            for all $0<\theta<1$ and $0<q\leq\infty$, we have
            \begin{align*}
                (X',Y')_{\theta,q}=\XY[\alpha][q],\quad
                \alpha:=(1-\theta)\theta_1+\theta\theta_2.
            \end{align*}
\end{itemize}
We cite some important results on the relationship between
interpolation, approximation and smoothness spaces.
To this end, an important tool are the so-called
\emph{Jackson (direct)}
\begin{align}\label{eq:jackson}
    \E{f}\leq Cn^{-\rJ}\snorm{f}[Y],\quad\forall f\in Y,
\end{align}
and \emph{Bernstein (inverse)}
\begin{align}\label{eq:bernstein}
    \norm{\varphi}[\Asq]\leq Cn^{\rB}\norm{\varphi}[X],\quad\forall\varphi\in\tool,
\end{align}
inequalities,
for some $\rJ>0$ and $\rB>0$.

\begin{theorem}[Interpolation and Approximation, {\cite[Chapter 7]{DeVore93}}, {\cite[Chapter 5]{Bennett}}]\label{thm:intapp}
    If the approximation class $\Asq(X)$ satisfies \ref{P1}, \ref{P3}, \ref{P4}
    and the space $Y$ satisfies the
    Jackson inequality \eqref{eq:jackson}, then
    \begin{align*}
        \XY[\alpha/\rJ][q]&\hookrightarrow\Asq(X),\quad 0<\alpha<\rJ,\;0<q<\infty,\\
        Y&\hookrightarrow\Asq[\rJ][\infty](X).
    \end{align*}

    If the approximation class $\Asq(X)$ satisfies \ref{P1} -- \ref{P6}
    and the space $Y$ satisfies
    the Bernstein inequality \eqref{eq:bernstein}, then
    \begin{align*}
        \Asq(X)&\hookrightarrow\XY[\alpha/\rB][q],\quad 0<\alpha<\rB,\\
        \Asq[\rJ][\infty](X)&\hookrightarrow Y.
    \end{align*}
\end{theorem}

\begin{theorem}[Interpolation and Smoothness, \cite{DeVore98}]\label{thm:intsmooth}
    The following identities hold:
    \begin{alignat*}{2}
        (\Lp,\Wkp)_{\theta,q}&=\Bsqp[\theta\alpha],&&\quad 0<\theta<1,\;0<q\leq\infty,\;1\leq p\leq\infty\\
        (\Bsqp[\alpha_1][q_1],\Bsqp[\alpha_2][q_2])_{\theta,q}
        &=\Bsqp[\alpha],&&\quad\alpha:=(1-\theta)\alpha_1+\theta \alpha_2,\;
        0<p,q,q_1,q_2\leq\infty\\
        (\Lp,\Bsqp[@][\tilde{q}])_{\theta,q}&=\Bsqp[\theta \alpha],&&\quad 0<\theta<1,\;0<p,q , \tilde{q}\leq\infty.
    \end{alignat*}
\end{theorem}


\subsection{Direct Estimates}

Theorems \ref{thm:intapp} and \ref{thm:intsmooth} allows us
to characterize the approximation classes introduced in
\Cref{sec:appspacestensors}
(and studied in \cite{partI}) by classical smoothness and interpolation spaces,
provided we can show for $X=\Lp$ and $Y=\Bsqp$
the Jackson \eqref{eq:jackson} and Bernstein \eqref{eq:bernstein}
inequalities.
As indicated earlier, the Jackson inequalities will follow from the
preparations in \Cref{sec:encoding}. We will also show
that Bernstein inequalities cannot hold.
This is an expression of the fact that the spaces $\Asq$
are ``too large'' in the sense that they are not continuously
embedded in any classical smoothness space.

We consider the approximation properties of tensor networks in $\Vb{m}$
for a fixed $m\in\N_0$.
Recall the definition of three different
complexity measures $\cost_{\mc N}$,
$\cost_{\mc C}$ and $\cost_{\mc S}$
from \eqref{eq:complex} and
the resulting approximating sets
$\tool^{\mc N}$, $\tool^{\mc C}$ and $\tool^{\mc S}$.
The best approximation error for $0<p\leq\infty$  is defined
accordingly as
\begin{align}\label{eq:besterror}
    E_n^{\mc N}(f)_p&:=\inf_{\varphi\in\Phi_n^{\mc N}}\norm{f-\varphi}[p],\\
        E_n^{\mc C}(f)_p&:=\inf_{\varphi\in\Phi_n^{\mc C}}\norm{f-\varphi}[p],\notag\\
            E_n^{\mc S}(f)_p&:=\inf_{\varphi\in\Phi_n^{\mc S}}\norm{f-\varphi}[p],\notag
\end{align}
and the corresponding approximation classes
$N^\alpha_q$, $C^\alpha_q$ and $S^\alpha_q$
as in \eqref{eq:approxclasses}.

\subsubsection{Sobolev Spaces}

We will apply local interpolation from \Cref{local-interpolation} to
approximate functions in Sobolev spaces $\Wkp[r][@]$ for any $r\in\N$.
These embeddings essentially correspond to
embeddings of Besov spaces $\Bsqp[@][p][p]$ into approximation
spaces $N^\alpha_q(\Lp)$,
$C^\alpha_q(\Lp)$ and $S_q^\alpha(\Lp)$: i.e., the approximation error is measured in the same
norm as smoothness\footnote{Compare to the embeddings
for RePU networks in \cite{Gribonval2019}.}.
To this end, we require

\begin{lemma}[Re-Interpolation]\label{lemma:reinterpol}
Let $f\in W^{  \bar m+1,p}$, $1\le p\le \infty$ and $\bar m \ge m$. For any $d\in \N_0$, $ \mathcal{I}_{b,d,  \bar m} f $ is a fixed knot spline in $ \S{N}{\bar m}{-1}$, $N=b^d$, and 
$$
\Vert f - \mathcal{I}_{b,d, \bar m} f \Vert_p \le C b^{-d( \bar  m +1)} \vert f \vert_{W^{  \bar m +1,p}}
$$
where $C$ is a constant depending only on $\bar m$ and $p$. Furthermore, for $\bar d \ge d$, 
$$
\Vert \mathcal{I}_{b,d,  \bar m} f  - \mathcal{I}_{b,\bar d,  m}\mathcal{I}_{b,d,  \bar m} f \Vert_p \le C'
\left(b^{-\bar d(m+1)}\vert f \vert_{W^{m +1,p}}
+b^{-(\bar d-d)(m+1)-d(\bar m+1)}\vert f \vert_{W^{\bar m +1,p}}\right)
$$
where $C'$ is a constant depending only on $\bar m$, $m$ and $p$.
\end{lemma}

\begin{proof}
See \Cref{proof:reinterpol}.
\end{proof}

With this we can show the direct estimate

\begin{lemma}[Jackson Inequality for Sobolev Spaces]\label{thm:jacksonfixed}
    Let $1\leq p\leq\infty$ and $r\in\N$.
    For any $f\in\Wkp[r]$ we have
    \begin{align}\label{eq:jacksonfixed2}
      E_n^{\mc N}(f)_p&\leq Cn^{-2r}\norm{f}[{\Wkp[r]}],\\
      E_n^{\mc S}(f)_p&\leq
      E_n^{\mc C}(f)_p\leq Cn^{-r}\norm{f}[{\Wkp[r]}],\notag
    \end{align}
    with constants $C$ depending on $r$, $m$, $b.$
\end{lemma}

\begin{proof}
    Let $N:=b^d$ and $r:=\bar m+1>m+1$
    and fix some $f\in\Wkp[r]$.
    The case $r\leq m+1$ can be handled similarly with fewer steps.
    Let $s:=\mathcal{I}_{b,d,  \bar m}f$ and
    $\tilde{s}:=\mathcal{I}_{ b,\bar d, m}s\in\Vbd[@][\bar d]{m}$
    with a $\bar d\geq d$ to be specified later.

    From \Cref{lemma:reinterpol} we have
    \begin{align}\label{eq:reinterpol}
        \norm{f-\tilde{s}}[p]\leq C_1\norm{f}[{\Wkp[r][@]}]
        \left(b^{-d r}+b^{-\bar d(m+1)}\right)
    \end{align}
    for a constant $C_1$ depending only on $r$, $m$ and $p$.
    Thus, we set
    \begin{align}\label{eq:dbarfixed}
        \bar d :=\left\lceil\frac{dr}{m+1}\right\rceil,
    \end{align}
    which yields
    \begin{align}\label{eq:errorsobolev}
        \norm{f-\tilde{s}}[p]\leq 2C_1\norm{f}[{\Wkp[r][@]}]N^{-r}.
    \end{align}
    From  
    \Cref{cor:fixedknotinterpolant} and \eqref{eq:dbarfixed}, we can estimate
    the complexity of $\tilde s$ as
    \begin{align*}
        n:=\cost_{\mc N}(\tilde{s})&\leq \bar C_2(\sqrt{N}+\log_b(N))
        \leq C_2\sqrt{N},\\
        n:=\cost_{\mc S}(\tilde{s})&\leq
        \cost_{\mc C}(\tilde{s})\leq
        \bar C_2(N+\log_b(N))\leq C_2N
    \end{align*}
    with a constant $C_2$ depending on $b$, $m$ and $\bar m$.
    Thus, inserting into \eqref{eq:errorsobolev}, we obtain
    \eqref{eq:jacksonfixed2}.
    For the case $\bar m\leq m$, the proof simplifies since
    we can represent $s$ exactly and use \Cref{cor:fixedknot}.
\end{proof}

\begin{remark}\label{rem:extendp}
    One could extend the statement of \Cref{thm:jacksonfixed} to the
    range $0<p<1$ by considering
    the Besov spaces $\Bsqp[@][p][p]$.
    For $r\leq m+1$, this can be done by
    using the characterization of Besov spaces $\Bsqp[@][p][p]$
    for $0<p\leq\infty$ by dyadic splines from \cite{Devore88},
    as was done in \cite[Theorem 5.5]{Gribonval2019}
    for RePU networks.
    For $r>m+1$, one would have to additionally replace
    the interpolation operator of \Cref{local-interpolation} with
    the quasi-interpolation operator from \cite{Devore88}.
\end{remark}

\Cref{thm:intapp} and \Cref{thm:jacksonfixed} imply
\begin{theorem}[Direct Embedding for Sobolev Spaces]\label{thm:sobolev}
For any $r\in\N$ and $1\leq p\leq\infty$, we have 
\begin{align*}
    W^{r,p} \hookrightarrow
    N^{2r}_\infty(\Lp),\quad
    W^{r,p} \hookrightarrow
    C^{r}_\infty(\Lp)\hookrightarrow S^{r}_\infty(\Lp),
\end{align*}
and for $0<q\leq\infty$
    \begin{alignat*}{3}
        &(\Lp,\Wkp[r])_{\alpha/2r,q}&&\hookrightarrow N^{\alpha}_q(\Lp),\quad &&0<\alpha<2r,\\
        &(\Lp,\Wkp[r])_{\alpha/r,q}&&\hookrightarrow
        C^{\alpha}_q(\Lp)\hookrightarrow        
        S^{\alpha}_q(\Lp),\quad &&0<\alpha<r.
    \end{alignat*}
\end{theorem}

\begin{corollary}
    Together with \Cref{thm:intsmooth}, this implies the statement
    of \Cref{thm:introembed}.
\end{corollary}

\subsubsection{Besov Spaces}\label{sec:subbesov}
Now we turn to direct embeddings of Besov spaces
$\Bsqp[@][\sobolev][\sobolev]$ into
$N^\alpha_q(\Lp)$, $C^\alpha_q(\Lp)$
and $S^\alpha_q(\Lp)$, where
$1/\sobolev=\alpha+1/p$. That is,
the smoothness is measured in a weaker norm
with $\tau < p$. The spaces $\Bsqp[@][\sobolev][\sobolev]$
are in this sense much larger than $\Bsqp[@][p][p]$.

They key to proving direct estimates for
Besov smoothness are the estimates of \Cref{cor:freeknot} and \Cref{cor:freeknotinterpolant} for
free knot splines. However, there are two issues
with encoding free knot splines as tensorized polynomials.
First, free knot splines are not restricted to $b$-adic knots
and thus cannot be represented exactly within $\Vb{m}$.
Second, even if all knots of a spline $s$ are $b$-adic,
the complexity of encoding $s$ as an element of $\Vb{m}$ depends
on the minimal level $d\in\N$ such that $s\in\Vbd{m}$, and
this level is not known in general.
We address these issues with the following two lemmas.

\begin{lemma}[$b$-adic Free Knot Splines]\label{lemma:badicfree}
    Let $0<p<\infty$, $0<\alpha<\m+1$ and let
    $\Sfreeb{N}{\bar m}$ denote the set of free knot splines of order $\m+1$ with
    $N+1$ knots restricted to $b$-adic points of the form
    \begin{align*}
        x_k:=i_kb^{-d_k},\quad 0\leq k\leq N,
    \end{align*}
    for some $d_k\in\N$ and $i_k \in \{0,\ldots,d_k\}$.
    For $\sobolev:=(\alpha+1/p)^{-1}$ being the Sobolev embedding number
    and $f\in\Bsqp[@][\sobolev][\sobolev]$, we have
    \begin{align}\label{eq:badicfree}
        \inf_{s\in\Sfreeb{N}{\bar m}}\norm{f-s}[p]\leq
        CN^{-\alpha}\snorm{f}[{\Bsqp[@][\sobolev][\sobolev]}].
    \end{align}
\end{lemma}

\begin{proof}
    See \Cref{proof:badicfree}.
\end{proof}

\begin{remark}
    In principle, Lemma \ref{lemma:badicfree} can be extended to the
    case $p=\infty$, $f\in\C[0]$ and the Besov space $\Bsqp[@][\sobolev][\sobolev]$
    replaced by the space of functions of bounded variation.
    However, the following Lemma \ref{lemma:smallinterval}
    does not hold for $p=\infty$, such that overall
    we can show the direct estimate of
    Theorem \ref{thm:jacksonbesov} only for $p<\infty$.
\end{remark}

\begin{lemma}[Smallest Interval Free Knot Splines]\label{lemma:smallinterval}
    Let $\delta>1$, $1\leq p<\infty$ and $f\in\Lp[p\delta]$.
    Let $q=q(\delta)>1$ be the conjugate of $\delta$ defined by
    \begin{align*}
        \frac{1}{\delta}+\frac{1}{q}=1.
    \end{align*}
    For $\varepsilon>0$, let $s=\sum_{k=1}^N s_k$ be a piece-wise polynomial
    such that
    \begin{align*}
        \norm{f-s}[p]\leq\varepsilon,
    \end{align*}
    where we assume $s_k$ is a polynomial over some interval $I_k$,
    zero otherwise and $I_k$, $k=1,\ldots,N$, form a partition of 
   $[0,1]$.
    
    Then, we can choose an index set $\Lambda=\Lambda(\varepsilon)
    \subset\set{1,\ldots,N}$ and a corresponding spline
    $\tilde{s}=\sum_{k \in \Lambda}^N\tilde{s}_k$  such that
    \begin{align}\label{eq:intervalbound}
        \norm{f-\tilde{s}}[p]\leq 2^{1/p}\varepsilon\quad\text{with}\quad
        |I_k|>N^{-q} {\norm{f}[p\delta]^{-pq}}\varepsilon^{pq}=:\varrho(\varepsilon),\quad
        k\in\Lambda.
    \end{align}
\end{lemma}

\begin{proof}
    See \Cref{proof:smallest}.
\end{proof}

\begin{remark}\label{remark:excess}
    We can guarantee $f\in\Lp[p\delta]$ by assuming
    excess regularity and using Sobolev embeddings as follows.
    Let $\alpha>0$, $0<p<\infty$, $\delta>1$ and
    $\sobolev:=(\alpha+1/p)^{-1}$.
    Defining $\alpha_\delta>\alpha$ as
    \begin{align*}
        \alpha_\delta:=\alpha+\frac{\delta-1}{p\delta},
    \end{align*}
    we get that the Sobolev embedding number for the
    combination $\alpha_\delta$, $p\delta$ is
    \begin{align*}
        \sobolev_\delta:=(\alpha_\delta+1/(p\delta))^{-1}=(\alpha+1/p)^{-1}=\sobolev.
    \end{align*}
    Then, assuming $f\in\Bsqp[\alpha_\delta][\sobolev][\sobolev]$
    implies $f\in\Lp[p\delta]$.
\end{remark}

\begin{lemma}[Jackson Inequality for $B^{\alpha}_{\tau,\tau}$]\label{thm:jacksonbesov}
    Let $1\leq p<\infty$, $0<\sobolev<p$, $\alpha>1/\sobolev-1/p$,
    and assume
    $f\in\Bsqp[\alpha][\sobolev][\sobolev]$.
    Then, for any $\sigma>0$,
    we obtain the direct estimates
    \begin{align}\label{eq:bestfreeknot}
            E_n^{\mc N}(f)_p&\leq C\snorm{f}[{\Bsqp[\alpha][\sobolev][\sobolev]}]
        n^{-\frac{\alpha}{1+\sigma}},\\
        E_n^{\mc C}(f)_p&\leq C\snorm{f}[{\Bsqp[\alpha][\sobolev][\sobolev]}]
        n^{-\frac{\alpha}{2+\sigma}},\notag\\
                E_n^{\mc S}(f)_p&\leq C\snorm{f}[{\Bsqp[\alpha][\sobolev][\sobolev]}]
        n^{-\frac{\alpha}{1+\sigma}},\notag
    \end{align}
    where the constants $C$ depend on $\alpha>0$,
    $\sigma>0$, $b$ and $m$. In particular, they diverge to infinity
    as $\sigma\rightarrow 0$ or
     $\alpha\rightarrow 1/\sobolev-1/p$.
\end{lemma}

\begin{proof}
    See \Cref{proof:lemmafree}.
\end{proof}

\begin{theorem}[{Direct Embedding for $\Bsqp[@][\sobolev][\sobolev]$}]\label{cor:besovspaces}
    Let $1\leq p<\infty$, $0<\sobolev<p$ and $r>1/\sobolev-1/p$.
    Then,
    for any $\sigma>0$,
    \begin{align*}
            \Bsqp[r][\sobolev][\sobolev]
            \hookrightarrow N^{r/(1+\sigma)}_\infty(\Lp),\quad
            \Bsqp[r][\sobolev][\sobolev]
            \hookrightarrow C^{r/(2+\sigma)}_\infty(\Lp),\quad
            \Bsqp[r][\sobolev][\sobolev]
            \hookrightarrow S^{r/(1+\sigma)}_\infty(\Lp),
    \end{align*}
    and
    \begin{alignat*}{2}
       (\Lp, \Bsqp[r][\sobolev][\sobolev])_{\alpha(1+\sigma)/r,q}
       &\hookrightarrow
       N_q^{\alpha}(\Lp),\quad &&0<\alpha<r/(1+\sigma),\\
       (\Lp, \Bsqp[r][\sobolev][\sobolev])_{\alpha(2+\sigma)/r,q}
       &\hookrightarrow
       C_q^{\alpha}(\Lp),\quad &&0<\alpha<r/(2+\sigma),\\
       (\Lp, \Bsqp[r][\sobolev][\sobolev])_{\alpha(1+\sigma)/r,q}
       &\hookrightarrow
       S_q^{\alpha}(\Lp),\quad &&0<\alpha<r/(1+\sigma).
    \end{alignat*}
\end{theorem}

\begin{proof}
    Follows from \Cref{thm:intapp}, \Cref{remark:excess}
    and \Cref{thm:jacksonbesov}.
\end{proof}

\subsubsection{Analytic Functions}

It is well known that analytic functions can be approximated
by algebraic polynomials with a rate exponential in the
degree of the approximating polynomials:
see, e.g., \cite[Chapter 7, Theorem 8.1]{DeVore93}.
In our setting, the polynomial degree in $\Vb{m}$ is fixed.
However, as before we can re-interpolate and
consider the corresponding approximation rate.
First, we show that polynomials can be approximated with an exponential
rate.

\begin{lemma}[Approximation Rate for Polynomials]\label{lemma:polys}
    Let $P\in\P_{\bar{m}}$ be an arbitrary polynomial with
    $\bar{m}>m$ (otherwise we have exact representation).
    Then, for $1\leq p\leq\infty$
    \begin{align*}
        E_n^{\mc N}(P)_p&\leq 
        C
        b^{-\frac{m+1}{(\bar{m}+1)}n}        
        \norm{P^{(m+1)}}[p],\\
       E_n^{\mc S}(P)_p&\leq E_n^{\mc C}(P)_p\leq 
        C
        b^{-\frac{m+1}{b(\bar{m}+1)^2}n}        
        \norm{P^{(m+1)}}[p],
    \end{align*}
    with $C$ independent of $\m$.
\end{lemma}

\begin{proof}
 See \Cref{proof:polys}.
\end{proof}

This implies analytic functions can be approximated
with an error decay of exponential type.
For the following statement we require
the distance function
\begin{align*}
    \dist(z, D):=\inf_{w\in D}|z-w|,\quad z\in\Cc,\quad D\subset\Cc.
\end{align*}

\begin{theorem}[Approximation Rate for Analytic Functions]\label{thm:analytic}
    Let $\rho>1$ and define
    \begin{align*}
        D_\rho:=
        \set{z\in\C:\dist(z, [0,1])<\frac{\rho-1}{2}}.
    \end{align*}
    Let $\rho:=\rho(f)>1$
    be such that $f:[0,1)\rightarrow\R$ has
    an analytic extension onto
    $D_\rho\subset\Cc$, but not onto any
    $D_{\tilde{\rho}}$ for $\tilde{\rho}>\rho$.
    Then,
    \begin{align}\label{eq:resultanal}
      E_n^{\mc N}(f)_\infty&\leq C[\min(\rho, b^{ {(m+1)}})]^{-n^{1/2}},\\
      E_n^{\mc S}(f)_\infty&\leq E_n^{\mc C}(f)_\infty\leq C[\min(\rho, b^{ {(m+1)}/b})]^{-n^{1/3}},\notag
    \end{align}
    where $C=C(f, m, b, \rho)$.
\end{theorem}

\begin{proof}
    See \Cref{proof:analytic}.
\end{proof}

\begin{remark}
    The above estimate can be further refined in the following ways:
    \begin{itemize}
        \item    The factor in the base of the exponent can be replaced by any
                    number $\theta$
                    \begin{align*}
                        \min(\rho, b^{(m+1)/b})<\theta<\max(\rho, b^{(m+1)/b}),
                    \end{align*}
                    with an adjusted constant $C$.
        \item    The inequality \eqref{eq:resultanal} can be stated
                    in the form as in \cite[Chapter 7, Theorem 8.1]{DeVore93}
                    to explicitly include the case $\rho=\infty$.
        \item    One can define classes of entire functions as
                    in \cite[Chapter 7, Theorem 8.3]{DeVore93} for a finer distinction of
                    functions that can be approximated with an exponential-type
                    rate.
        \item    One can extend the result to approximation of analytic functions
                    with singularities applying similar ideas as in \cite{Kazeev2017}.
    \end{itemize}
\end{remark}


\subsection{Inverse Estimates}

It is well known in tensor approximation of high-dimensional
functions and approximation with
neural networks (see \cite{Gribonval2019}) that highly
irregular functions can in some cases be
approximated or even represented exactly with low or constant rank or
complexity\footnote{Think of a rank-one tensor product of jump functions.}.
This fact is reflected in the lack of inverse estimates
for tensorized approximation of one-dimensional functions
as the next statement shows.

\begin{theorem}[No Inverse Embedding]\label{thm:noinverse}
    For any $\alpha>0$, $0<p,q\leq\infty$ and any $\tilde{\alpha}>0$
    \begin{align*}
      C^\alpha_q(\Lp)\not\hookrightarrow\Bsqp[\tilde{\alpha}].
    \end{align*}
\end{theorem}

\begin{proof}
    For ease of notation we restrict ourselves to
    $b=2$, but the same arguments apply
    for any $b\geq 2$.
    The proof boils down to finding a counterexample of
    a function that can be efficiently represented within $\Vb{m}$
    but has ``bad'' Besov regularity. To this end,
    we use the \emph{sawtooth} function, see
    \cite{telgarsky2015} and \Cref{fig:sawtooth}.
    
    \begin{figure}[h]
        \centering
        \includegraphics[scale=.5]{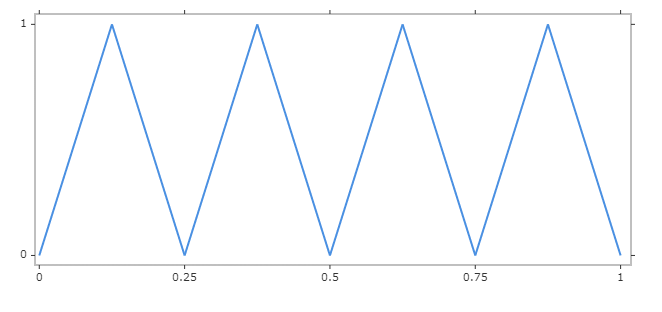}
        \caption{``Sawtooth'' function.}
        \label{fig:sawtooth}
    \end{figure}
    
    Specifically, consider the linear functions
    \begin{align*}
        \psi_1(y):=y,\quad\psi_2(y):=1-y,\quad 0\leq y<1.
    \end{align*}
    For arbitrary $d\in\N$, set
    \begin{align*}
        \tensor{\varphi}_d(i_1,\ldots,i_d,y):= 
        \delta_{0}(i_d)\psi_1(y)+\delta_{1}(i_d)\psi_2(y).
    \end{align*}
    Then, $\varphi_d = T_{b,d}^{-1}  \tensor{\varphi}_d \in\Vbd[2]{m}$ with $r_\nu(\varphi_d) = 2$ for all $1\leq \nu\leq d$.
    Thus,
    \begin{align}\label{eq:phicost}
       \cost_{\mc C}(\varphi_d)\leq 8d+2m +2. 
    \end{align}
    
    We can compute the $\Lp$-norm of $\varphi_d$ as
    \begin{align}\label{eq:philpnorm}
        \norm{\varphi_d}[p]^p=2^d\int_0^{b^{-d}}(2^dy)^p\d y=\frac{1}{p+1}.
    \end{align}
    
    Next, since $C^\alpha_q(\Lp)$ satisfies \ref{P1} -- \ref{P4}, this implies 
    $C^\alpha_q(\Lp)$ satisfies the Bernstein inequality
    (see \cite[Chapter 7, Theorem 9.3]{DeVore93})
    \begin{align}\label{eq:phibern}
        \norm{\varphi_d}[C^\alpha_q]\leq Cn^\alpha\norm{\varphi_d}[p],\quad
        \forall\varphi\in\Phi_n.
    \end{align}
    On the other hand, by \cite[Lemma 5.12]{Gribonval2019},
    \begin{align}\label{eq:phibesov}
        \norm{\varphi_d}[{\Bsqp[\tilde{\alpha}]}]\geq c2^{\tilde{\alpha}d},
    \end{align}
    for any $\tilde{\alpha}>0$.

    Assume the Bernstein inequality holds in $\Bsqp[\tilde{\alpha}]$ for some
    $\tilde{\alpha}>0$. For $n\in\N$ large enough,
    let $d:=\lfloor n/8-m/4-1/4\rfloor\geq 2$.
    Then, by \eqref{eq:phicost}, $\varphi_d\in\Phi^{\mc C}_n$.
    By \eqref{eq:phibern} and \eqref{eq:phibesov},
    \begin{align*}
        Cn^\alpha\norm{\varphi_d}[p]\geq \norm{\varphi_d}[{C^\alpha_q}]
        \gtrsim \norm{\varphi_d}[{\Bsqp[\tilde{\alpha}]}]\gtrsim
        2^{\tilde{\alpha}d}\gtrsim
        2^{\frac{\tilde{\alpha}}{8}n}.
    \end{align*}
    Together with \eqref{eq:philpnorm}, this is a contradiction and thus
    the claim follows.
\end{proof}

In Section \ref{sec:encoding}, we demonstrated that when representing
classical tools with the tensorized format we obtain
a complexity
that is
similar (or slightly worse) than for the corresponding classical
representation.
This reflects the fact that these tools are tailored for approximation
in classical smoothness spaces and we therefore cannot expect
better ``worst case'' performance in these spaces.
This was also observed in high-dimensional approximation,
see \cite{Schneider2014}.

On the other hand, the above theorem demonstrates
that tensor networks are efficient for functions that cannot
be captured with classical smoothness (see also \cite{Ali2019}).
The cost $n$ in $\Phi^{\mc C}_n$ is comprised of both
the discretization level $d$ and the tensor ranks $r_\nu$
that, in a sense, reflect algebraic properties of
the target function.

The proof of Theorem \ref{thm:noinverse} shows that
tensor networks  are particularly effective in approximating functions
with a high degree of self-similarity.
Such functions do not have to possess any smoothness
in the classical sense.
The ranks reflect global algebraic features,
while smoothness reflects local ``rate of change''
features.

However, one would expect that, if one enforces
a full-rank structure or, equivalently, limits the depth of the
corresponding tensor network,
we should recover inverse estimates
similar to classical tools
from Section \ref{sec:encoding}.

\begin{proposition}[Inverse Embedding for Restricted  $\Phi^{\mc N}_n$]\label{prop:inverse}
    Let $1\leq p<\infty$.
    Define for $n\in\N$, $\rB>0$ and $\CB>0$ the restricted sets  
    \begin{align*}
        \PhinB:=\set{\varphi\in\Vb{m}:
        \cost_{\mc N}(\varphi)\leq n\quad\text{and}\quad
        d(\varphi)\leq\rB\log_b(n)+\CB}.
    \end{align*}
    Then,
    \begin{enumerate}[label=(\roman*)]
        \item    $\PhinB$ satisfies \ref{P1} -- \ref{P6}
                    and thus $\Asq(\Lp, (\PhinB))$ are quasi-normed
                    linear spaces satisfying direct and inverse estimates.
        \item\label{inverserestricted}    The following inverse estimate holds:
                    \begin{align*}
                        \snorm{\varphi}[{\Bsqp[m+1][\sobolev][\sobolev]}]&\leq C\norm{\varphi}[p]b^{\CB(m+1)}n^{\rB(m+1)},
                    \end{align*}
                    for any $\varphi\in\PhinB$, where $\sobolev>0$
                    is the Sobolev embedding number.
        \item    We have the continuous embeddings
                    \begin{alignat*}{2}
                        \Asq(\Lp,(\PhinB))&\hookrightarrow
                        (\Lp,\Bsqp[m+1][\sobolev][\sobolev])_{\frac{\alpha}{\rB(m+1)},q},&&\quad
                        0<\alpha<\rB(m+1),\\
                        \Asq[\rB(m+1)][\infty](\Lp,(\PhinB))&\hookrightarrow
                        \Bsqp[m+1][\sobolev][\sobolev].
                    \end{alignat*}
    \end{enumerate}
\end{proposition}

\begin{proof}
    The restriction on $\PhinB$ ensures functions
    such as the sawtooth function from \Cref{fig:sawtooth}
    are excluded.
    \begin{enumerate}[label=(\roman*)]
        \item \ref{P1} -- \ref{P3} is trivial.
                    For \ref{P4}: since $\Phi^{\mc N}_n+\Phi^{\mc N}_n\subset\Phi^{\mc N}_{cn}$ and
                    \[d(\varphi_1+\varphi_2)\leq\max(d_1,d_2)\leq
                    \rB\log_b(n)\leq\rB\log_b(cn)\] for $\varphi_1,\varphi_2\in\PhinB$,
                    then \ref{P4} holds for $\PhinB$
                    for the same $c$.
                    For \ref{P5}:
                    $\bigcup_{n=0}^{\infty}\Phi^{\mc N}_n=\bigcup_{n=0}^\infty
                    \PhinB$ and thus density follows as in
                    \cite[Theorem 2.25]{partI}.
                    Finally, \ref{P6} follows as in \cite[Lemma 3.14]{partI}.
        \item    Any $\varphi\in\PhinB$ is a spline with at most
                    $b^{d(\varphi)}\leq b^{\CB}n^{\rB}$ pieces.
                    Thus, we can use classical inverse estimates to obtain
                    the inequality.
        \item    Follows from \ref{inverserestricted} and
                    Theorem \ref{thm:intapp}.
    \end{enumerate}
\end{proof}
\section{The Roles of Depth and Sparse Connectivity}\label{sec:conclusion}
    One could ask how the direct estimates would change if we replace
    $\tool^{\mc N}$ with $\PhinB$ from \Cref{prop:inverse}.
    Strictly speaking, this would require
    lower bounds for the complexity $n:=\cost_{\mc N}(\varphi)$.
    Nonetheless, a simple thought experiment reveals some key features
    of $\PhinB$, assuming the upper bounds for $n$ in this section are sharp
    to some degree.
    
    Consider the case of Sobolev spaces  $\Wkp[r]$  from \Cref{thm:jacksonfixed}
    with $r\leq m+1$. Then,
    assuming the upper bounds from \Cref{thm:jacksonfixed} are sharp,
    we have
    \begin{align*}
        n\sim C_1(b,m)b^d.
    \end{align*}
    I.e., the approximands of \Cref{thm:jacksonfixed} satisfy
    $\varphi\in\PhinB$ for $\rB=1$ and $\CB=\CB(b, m)$.
    Hence, in this case we would indeed obtain the same
    approximation rate as with $\tool$, in
    addition to inverse estimates from \Cref{prop:inverse}.
    
    Consider now $\Wkp[r]$ with $r>m+1$.
    In this case we have $\rB=\rB(r)>1$ with
    $\rB\rightarrow\infty$ as $r\rightarrow\infty$.
    In other words, if we fix $\PhinB$ with some $\rB>1$,
    then we would obtain direct estimates
    for $\Wkp[r]$
    as in \Cref{thm:jacksonfixed},
    with $0<r\leq \bar r$ for $\bar r$ depending on $\rB>1$.
    I.e., $\bar r=m+1$ for $\rB=1$ and $\bar r\rightarrow\infty$
    as $\rB\rightarrow\infty$.
    
   Finally, consider the direct
    estimate for Besov spaces  $\Bsqp[\alpha_\delta][\sobolev][\sobolev]$
    from \Cref{thm:jacksonbesov}.
    Again, assuming the upper bounds of this lemma are sharp
    and $\alpha<m+1$,
    we would obtain
    \begin{align*}
        n \sim CNd,
    \end{align*}
    where $N$ is the number of knots of a corresponding free knot spline
    and $d$ is the maximal level of said spline.
    From \Cref{lemma:smallinterval}, we could assume
    $d\sim\log(N)$ and in this case
    \begin{align*}
        d\sim\log(N)\lesssim\log(N)+\log\log(N)\lesssim\log(n),
    \end{align*}
    in which case we claim we could recover direct estimates as in
    \Cref{thm:jacksonbesov}. However, note that, in order to recover
    near to optimal rates, we would have to consider
    the complexity measure $\cost_{\mc S}$ (or $\cost_{\mc N}$) -- i.e.,
    we have to account for sparsity.
    And, as for Sobolev spaces, for $\alpha\geq m+1$,
    $\PhinB$ is not sufficient anymore as we require depth (and sparsity).
    
    Thus, when comparing approximation with tensor networks
    to approximation with classical tools, we see that depth can
    very efficiently replicate approximation with higher-order
    polynomials: that is, with exponential convergence.
    It was already noted in \cite{Grasedyck2010, Oseledets2012} that
    (deep) tree tensor networks can represent polynomials with
    bounded rank, while the canonical (CP) tensor format,
    corresponding to a shallow network,
    can only do so approximately with ranks bounded
    by the desired accuracy.
    Moreover, similar observations about depth and polynomial degree
    were made about ReLU networks, see, e.g.,
    \cite{Yarotsky2017, Opschoor2020, Opschoor201907}.
    
    On the other hand, sparse connectivity is necessary to recover
    classical adaptive (free knot spline) approximation, see
    \Cref{cor:besovspaces}. In other words:
    sparse tensor networks can replicate $h$-adaptive
    approximation, while deep tensor networks can replicate $p$-adaptive
    approximation, and, consequently, sparse \emph{and} deep
    tensor networks can replicate $hp$-adaptive approximation.

\bibliographystyle{abbrv}
\bibliography{literature}

\appendix

\newpage
\section{Proofs for \Cref{sec:encoding}}
  \begin{proof}[Proof of \Cref{cor:freeknot}]
  The bounds for $\cost_{\mc N}(\varphi)$ and $\cost_{\mc C}(\varphi)$ directly follow from \Cref{thm:freeknot}. 
To obtain the bound on the sparse representation complexity, we have to provide a representation of $\varphi$ in a tensor format. First, we note that the interval $I_k = [x_{k-1}^b,x_{k}^b)$ is such that $I_k = \cup_{i=1}^{n_k} I_{k,i}$, where the $I_{k,i}$ are $n_k$ contiguous intervals from $b$-adic partitions of $[0,1)$, and the minimal $n_k$ can be bounded as $n_k \le 2d(b-1)$.
To illustrate why this bound holds, we refer to \Cref{fig:minpath}.

If $d$ is the maximal level, the subsequent partitioning of $[0,1)$
for levels $l=0, 1, 2,\ldots, d$ can be represented as a tree, where
each vertex has $b$ sons, i.e., each interval is subsequently split
into $b$ intervals. Then, the end-points $x^b_{k-1}$ and
$x^b_k$ of an arbitrary interval $I_k$ correspond to two points
in this interval partition tree. The task of finding a minimal sub-partitioning
$I_k = \cup_{i=1}^{n_k} I_{k,i}$ is then equivalent to
finding the shortest path in this tree, and $2d$ represents the longest possible path.

\begin{figure}[h]
    \centering
    \includegraphics[scale=.4]{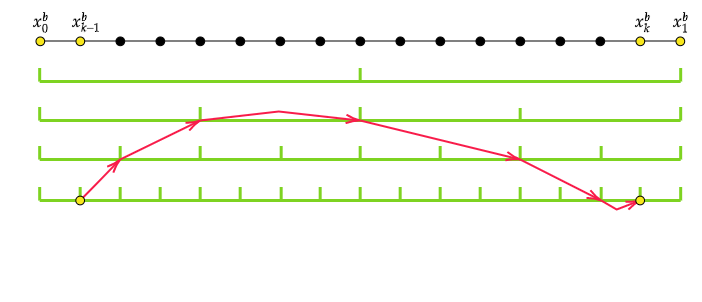}
    \caption{Visual representation of different partitioning levels of
    the interval $[0,1)$, with $b=2$ and $d=4$.}
    \label{fig:minpath}
\end{figure} 

In \Cref{fig:minpath}, we depict a scenario close to the ``worst case''.
In order to reach vertex $x_k^b$ from vertex $x_{k-1}^b$,
at most, we would have to traverse the tree up (towards the root) and back down.
On each level, we would need at most $b-1$ horizontal steps.
Thus, we require at most $2d(b-1)$ steps to reach $x_k^b$.

Then, $\varphi$ admits a representation as $\varphi = \sum_{k=1}^N \sum_{i=1}^{n_k} s_{i,k}$, with $s_{i,k}$ supported on $I_{i,k}$ and polynomial on this interval. Let $\lambda := (k,i)$. We have $I_\lambda = [b^{d_\lambda} j_\lambda , b^{d_\lambda}(j_{\lambda}+1))$ for some $d_\lambda \le d$ and $j_\lambda \in \{0,\hdots,b^{d_\lambda-1}-1\}$. By denoting $(j_{\lambda,1},\hdots,j_{\lambda,d_{\lambda}})$ the representation of $j_\lambda$ in base $b$, $s_\lambda$ admits a tensorization 
  $$
 T_{b,d_\lambda} (s_\lambda) = \delta_{j_{\lambda,1}} \otimes \hdots  \otimes \delta_{j_{\lambda,d_{\lambda}}} \otimes  p_\lambda,
  $$   
  with $p_\lambda\in \P_{  m}$, so that $\cost_{\mathcal{S}}(s_\lambda) \le d_\lambda +   \dim S$. From \cite[Lemmas 3.11 and 3.12]{partI}, we deduce that 
  $$
  \cost_{\mathcal{S}}(\varphi) \le \sum_{k=1}^N \sum_{i =1}^{n_k}  
  b (d_{k,i} +   \dim S) + b^2 (\dim S)^3 (d-d_{k,i}) \le 2 b^2 (\dim S)^3 d \sum_{k=1}^N n_k \le 4 b^3 (m+1)^3 d^2 N.
  $$
\end{proof}

 \begin{proof}[Proof of \Cref{cor:freeknotinterpolant}]
From  \Cref{lem:boundranksfreeknotsplines}, we have 
\begin{align*}
\cost_{\mathcal{N}}(\interpol[b,\bar d][m](\varphi)) &\le  d (\bar m + N)+ (\bar d-d) (\bar m+N) \le (\bar m+1) d N + (\bar d-d) (\bar m+N),\\
\cost_{\mathcal{C}}(\interpol[b,\bar d][m](\varphi)) &\le b (\bar m + N) + (d-1) b  (\bar m+N)^2 + (\bar d-d) b(\bar m+1)^2 + b (m+1) \\
&\le 2b d (\bar m+1)^2 d N^2 +   (\bar d-d) b (\bar m+1)^2.
\end{align*}
Now we consider the sparse representation complexity.
The function $\varphi$ admits a representation
  $\varphi = R_{b,d,\bar m,\bs r}(\mathbf{v})$ for some $\bs r\in \N^r$ and a tensor network $\mathbf{v} \in \mathcal{P}_{b,d,\bar m , \bs r}$ such that $\cost_{\mathcal{S}}(\mathbf{v}) = \cost_{\mathcal{S}}(\varphi)  $
 and 
\begin{align*}
&T_{b,d}(\varphi)(i_1,\hdots,i_d,y) = \sum_{k_1=1}^{r_1}\cdots\sum_{k_d=1}^{r_d}\sum_{q=1}^{\bar m+1} v_{1}^{k_1}(i_1)   \cdots v_d^{k_{d-1},k_d}(i_d) v^{k_d,q}_{d+1} \varphi_{q}^{\bar m+1}(y),
\end{align*}
with the $\varphi_q^{\bar m+1}$ forming a basis of $\P_{\bar m}$.  From \Cref{cor:fixedknot}, we know that $\cost_{\mathcal{S}}(\varphi)  \le C_1N$ for some constant $C_1$ depending only on $b$ and $\bar m$. 
Then from 
\eqref{tensor-structure-local-projection}, 
we have that 
\begin{align*}
&T_{b,\bar d}(\interpol[b,\bar d][m](\varphi))(i_1,\hdots,i_{\bar d},y)  \\
&=\sum_{k_1=1}^{r_1}\cdots\sum_{k_d=1}^{r_d}\sum_{q=1}^{\bar m+1} v_{1}^{k_1}(i_1)   \cdots v_d^{k_{d-1},k_d}(i_d) v^{k_d,q}_{d+1} T_{b,\bar d - d}(\interpol[b,\bar d-d][m](\varphi_q^{\bar m+1}))(i_{d+1},\hdots,i_{\bar d},y)\\
&=\sum_{k_1=1}^{r_1}\cdots\sum_{k_d=1}^{r_d}\sum_{q=1}^{\bar m+1}\sum_{j_{d+1}=1}^{b} v_{1}^{k_1}(i_1)   \cdots v_d^{k_{d-1},k_d}(i_d) \bar v^{k_d,(q,j_{q+1})}_{d+1}(i_{d+1}) T_{b,\bar d - d}(\interpol[b,\bar d-d][m](\varphi_q^{\bar m+1}))(j_{d+1},\hdots,i_{\bar d},y),
\end{align*}
with $ \bar v^{k_d,(q,j_{q+1})}_{d+1}(i_{d+1})  = v^{k_d,q}_{d+1} \delta_{j_{d+1}}(i_{d+1})$ such that 
$\Vert \bar v_{d+1} \Vert_{\ell^0} =  b \Vert v_{d+1} \Vert_{\ell^0}$. 
Noting that $r_\nu(\interpol[b,\bar d-d][m](\varphi_q^{\bar m+1})) \le r_\nu(\varphi_q^{\bar m+1}) \le \bar m+1$ for all $\nu\in \N$, and 
following the proof of \cite[Lemma 3.11]{partI}, we can prove that for $\bar d-d \ge 2$, $\interpol[b,\bar d-d][m](\varphi_q^{\bar m+1})$ admits a representation 
 \begin{align*}
 &T_{b,\bar d-d}(\interpol[b,\bar d-d][m](\varphi_q))(j_{d+1},i_{d+2},\hdots,i_{\bar d},y) \\
& = \sum_{\alpha_{2},q_2=1}^{\bar m+1} \hdots  
\sum_{\alpha_{l},q_{l}=1}^{\bar m+1}  \sum_{p=1}^{m+1} \bar v^{(q,j_{d+1}),(q_2,\alpha_2)}_{d+2}(i_{d+2})   \hdots  \bar v^{(q_{l-1},\alpha_{l-1}),(q_l,\alpha_l)}_{\bar d}(i_{\bar d}) 
 \bar v_{\bar d+1}^{(q_l,\alpha_l),p} \varphi_{p}(y) 
    \end{align*}
  with the $\varphi_p$ forming a basis of $\P_m$ and  with $\bar v_{d+2} \in \R^{b\times (b(\bar m+1)) \times (\bar m+1)^2}$, $\bar v_\nu \in \R^{b\times (\bar m+1)^2\times (\bar m+1)^2}$   for $d+3\le \nu\le \bar d$,  and $\bar v_{\bar d+1} \in \R^{(\bar m+1)^2\times (m+1)}$. Then, we have 
$\interpol[b,\bar d][m](\varphi) = R_{b,\bar d,m,\overline{\bs r}}(\overline{\mathbf{v}})$ with $\overline{\mathbf{v} }= (\bar v_1,\hdots,v_{d},\bar v_{d+1}, \hdots, \bar v_{\bar d+1})$ such that  
\begin{align*}
\cost_{\mathcal{S}}(\interpol[b,\bar d][m](\varphi)) &\le b \cost_{\mathcal{S}}(\varphi) +  b^2 (\bar m+1)^3 + 
b (\bar m+1)^4 (\bar d - d -2) + (\bar m+1)^2 (m+1)\\
&\le \max\{b,m+1\} ( \cost_{\mathcal{S}}({\varphi})  + b (\bar m+1)^3 + (\bar m+1)^4 (\bar d - d -2) + (\bar m+1)^2) \\
&\le  \max\{b,m+1\}(\cost_{\mathcal{S}}({\varphi}) + b  (\bar m+1)^4 (\bar d - d) ).
\end{align*}
For $\bar d - d =1,$ we have the representation\begin{align*}
 T_{b,\bar d-d}(\interpol[b,\bar d-d][m](\varphi_q^{\bar m+1}))(j_{d+1},y)  =  \sum_{p=1}^{m+1}  \bar v_{d+2}^{(q,j_{d+1}),p} \varphi_p(y)
    \end{align*}
    with some $\bar v_{d+2} \in \R^{(b (\bar m+1)) \times (m+1)}$. Then for $\bar d - d =1$, $\varphi \in R_{b,\bar d,S,\overline{\bs r}}(\overline{\mathbf{v}})$ with $\overline{\mathbf{v} }= (\bar v_1,\hdots,v_{d},\bar v_{d+1},\bar v_{\bar d+2}) $, and 
 \begin{align*}
    \cost_{\mathcal{S}}(\interpol[b,\bar d][m](\varphi))& \le  b  \cost_{\mathcal{S}}({\mathbf{v}}) +  b (\bar m+1) (m+1)
    \le \max\{b,m+1\}( \cost_{\mathcal{S}}({\mathbf{v}})  +b (\bar m+1) (\bar d-d)).
  \end{align*}
  Finally for $\bar d=d$, we simply have $\interpol[b,\bar d-d][m] = \mathcal{I}_m$, and we can show that $\interpol[b,\bar d][m](\varphi) = R_{b,d,m,\bs r}(v_1,...,v_d,\bar v_{d+1})$ with $\Vert \bar v_{d+1} \Vert_{\ell^0}  \le (m+1)\Vert v_{d+1} \Vert_{\ell^0}  $, so that   
  \begin{align*}
    \cost_{\mathcal{S}}(\interpol[b,\bar d][m](\varphi)) \le (m+1) \cost_{\mathcal{S}}({\mathbf{v}}) .
      \end{align*}
      Then for  any $\bar d \ge d$, we have 
      \begin{align*}
      \cost_{\mathcal{S}}(\interpol[b,\bar d][m](\varphi)) \le \max\{b,m+1\} ( \cost_{\mathcal{S}}({\varphi}) + b(\bar m+1)^4 (\bar d - d )),
    \end{align*}      
    and we conclude by using  $\cost_{\mathcal{S}}({\varphi})\le C_1N.$
\end{proof}

\section{Proofs for \Cref{sec:dirinv}}
\begin{proof}[Proof of \Cref{lemma:reinterpol}]\label{proof:reinterpol}

From \eqref{tensor-structure-local-projection}, we know that $s :=  \mathcal{I}_{b,d, \bar  m} f $ admits a tensorization $  \tensor{s} := T_{b,d} s = (id_{\{1,\hdots,d\}} \otimes \mathcal{I}_{\bar m}) \tensor{f}$, with $\tensor{f} = T_{b,d} f$. Then
$$T_{b,d} (f-s) =   \sum_{j\in \Ib^d} \delta_{j_1} \otimes \hdots \otimes \delta_{j_d} \otimes ( g_j - \mathcal{I}_{\bar m} g_j),  
$$
with $g_j =  \tensor{f}(j_1,\hdots,j_d,\cdot)$. Using the property \eqref{property-local-interpolation} of operator $\mathcal{I}_{\bar m}$, with a constant $  C$ depending on $\bar m$ and $p$, we have 
$$
 \Vert g_j - \mathcal{I}_{\bar m} g_j \Vert_p \le   C \vert   g_j \vert_{W^{\bar m+1,p}} =   C \Vert   D^{\bar m+1} g_j \Vert_{p}
.
$$
Then, using  \Cref{thm:tensorizationmap}, we have for $p<\infty$
\begin{align*}
\Vert f - s \Vert_p^p &=  \sum_{j \in \Ib^d} b^{-d} \Vert g_j - \mathcal{I}_{\bar m} g_j \Vert_p^p \le   C^p \sum_{j \in \Ib^d} b^{-d} \Vert   D^{\bar m+1} g_j \Vert_{p}^p= C^p \Vert (id_{\{1,\hdots,d\}} \otimes D^{\bar m+1}) \tensor{f} \Vert_p^p
\end{align*}
and 
\begin{align*}
\Vert f - s \Vert_\infty &=  \max_{j \in \Ib^d} \Vert g_j - \mathcal{I}_{\bar m} g_j \Vert_\infty \le  C  \max_{j \in \Ib^d} 
\Vert   D^{\bar m+1} g_j \Vert_{\infty}
\le    C   \Vert (id_{\{1,\hdots,d\}} \otimes D^{\bar m+1}) \tensor{f} \Vert_\infty.
\end{align*}
Then, from \cite[Theorem 2.15]{partI}
we deduce 
\begin{align*}
\Vert f - s \Vert_p &\le   C \Vert (id_{\{1,\hdots,d\}} \otimes D^{\bar m+1}) \tensor{f} \Vert_p =  C b^{-d(\bar m+1)} \Vert f \Vert_{W^{\bar m+1,p}}.
\end{align*}
For $\bar d\ge d$, we obtain from
\cite[Lemma 2.6]{partI} and \eqref{tensor-structure-local-projection} 
that $T_{b,  d} \mathcal{I}_{b,\bar d,m} T_{b,  d}^{-1} = T_{b,  d} T_{b,\bar d}^{-1}  (id_{\{1,\hdots,\bar d\}} \otimes \mathcal{I}_m ) T_{b,\bar d} T_{b,  d}^{-1}  = 
id_{\{1,\hdots,d\}}\otimes (T_{b,\bar d-d} (id_{\{1,\hdots,\bar d-d\}}\otimes \mathcal{I}_m)T_{b,\bar d-d}^{-1} ) = id_{\{1,\hdots,d\}}\otimes \mathcal{I}_{b,\bar d-d,m}$. Then  
 $\tilde s :=  \mathcal{I}_{b,\bar d,  m} s $ admits for tensorization   $  T_{b,  d} \tilde s =   (id_{\{1,\hdots,  d\}} \otimes \mathcal{I}_{ b,\bar d - d,m}) \tensor{s} = 
\sum_{j\in \Ib^d}\delta_{j_1}\otimes \hdots \otimes \delta_d \otimes (\mathcal{I}_{ b,\bar d - d,m} \mathcal{I}_{\bar m} g_j),$ which yields 
$$T_{b,d} (s-\tilde s) =   \sum_{j\in \Ib^d} \delta_{j_1} \otimes \hdots \otimes \delta_{j_d} \otimes (  \mathcal{I}_{\bar m} g_j -  \mathcal{I}_{ b,\bar d - d,m} \mathcal{I}_{\bar m} g_j ).
$$
From the property \eqref{property-local-interpolation} of  $\mathcal{I}_{m}$, with a constant $\tilde C$ depending on $m$ and $p$, and the property of $\mathcal{I}_{\bar m}$,
we obtain 
$$
\Vert \mathcal{I}_{\bar m} g_j -  \mathcal{I}_{ b,\bar d - d,m} \mathcal{I}_{\bar m} g_j \Vert_p \le  \tilde C b^{-(\bar d -d) (m+1)} \vert \mathcal{I}_{\bar m} g_j \vert_{W^{m+1,p}}\leq \tilde  C b^{-(\bar d -d) (m+1)}\left(
\snorm{g_j}[{\Wkp[m+1]}]+  C\snorm{g_j}[{\Wkp[\bar m+1]}]\right).
$$
In the same way as above, we deduce 
\begin{align*}
\Vert s - \tilde s \Vert_p &\le   \tilde  C b^{-(\bar d -d) (m+1)} 
\left(\Vert (id_{\{1,\hdots,d\}} \otimes D^{m+1}) \tensor{f} \Vert_p
+  C\Vert (id_{\{1,\hdots,d\}} \otimes D^{\bar m+1}) \tensor{f} \Vert_p\right)
\\
&=      C' b^{-(\bar d-d) (m+1)}(b^{-d(m+1)}\snorm{f}[{\Wkp[m+1][p]}]+ b^{-d(\bar m+1)}\snorm{f}[{\Wkp[\bar m+1][p]}]), 
\end{align*}
with $C' = \tilde C \max\{1,C\}$ depending on $m$, $\bar m$ and $p$, which completes the proof. 
\end{proof}

\begin{proof}[Proof of \Cref{lemma:badicfree}]
\label{proof:badicfree}
    The proof is a modification of the proof of Petrushev for free
    knot splines (see \cite[Chapter 12, Theorem 8.2]{DeVore93}).
    The first step is the optimal selection of $n$ intervals
    that, in a sense, balances out the Besov norm
    $\snorm{f}[{\Bsqp[@][\sobolev][\sobolev]}]$.
    In this step, unlike in the case of classic free knot splines,
    we are restricted to $b$-adic knots.
    The second step is a polynomial approximation
    over each interval and is essentially the same as with free knot splines.
    We demonstrate this step here as well  for completeness.
    
    First, we define a set function that we will use for the selection of
    the $n-1$ $b$-adic knots.  Let $r:= \lfloor\alpha\rfloor +1$ and 
    \begin{align*}
        M^\sobolev:=
        \int_0^1t^{-\alpha\sobolev-1}\w_{r}(f,t)_\sobolev^\sobolev\d t,
    \end{align*}
    where   $\w_{r}$ is the \emph{averaged modulus of smoothness},
    i.e.,
    \begin{align*}
        \w_{r}(f,t)_\tau^\tau :=\frac{1}{t}\int_0^t\norm{\diff_h^{r}[f]}[\tau]^\tau\d h.
    \end{align*}
    By \cite[Chapters 2 and 12]{DeVore93}, $M$ is equivalent to
    $\snorm{f}[{\Bsqp[@][\sobolev][\sobolev]}]$.
    
    Let  
    \begin{align*}
        g(x,h,t):= \begin{cases} t^{-\alpha\sobolev-2} |\diff_h^{r}[f](x)|^{\tau} & \text{if $h \in  [0,t]$ and $x \in [0, 1- rh] $},  
        \\ 0 &\text{elsewhere} .  
        \end{cases}
    \end{align*}
    Then,
    \begin{align*}
        M^\sobolev&=\int_0^1\int_0^\infty\int_0^1g(x,h,t)\d x\d h\d t=
        \int_0^1G(x)\d x,\\
        G(x)&:=\int_0^\infty\int_0^1g(x,h,t)\d h\d t.
    \end{align*}
    The aforementioned set function is then defined as
    \begin{align*}
        \setf(t):=\int_0^tG(x)\d x.
    \end{align*}
    This function is positive, continuous and monotonically increasing with
    \begin{align*}
        \setf(0)=0\quad\text{and}\quad\setf(1)=M^\sobolev
        \sim\snorm{f}[{\Bsqp[@][\sobolev][\sobolev]}]^\sobolev.
    \end{align*}
    Thus, we can pick $N$ intervals $I_k$, $k=1,\ldots,N$, with disjoint interiors
    such that
    \begin{align*}
        \bigcup_{k=1}^NI_k=[0,1]\quad\text{and}\quad
        \int_{I_k}G(x)\d x=\frac{M^\sobolev}{N}.
    \end{align*}
    
    This would have been the optimal knot selection for free knot splines.
    For our purposes we need to restrict the intervals to $b$-adic knots.
    More precisely, we show that with restricted intervals
    we can get arbitrarily close to the optimal choice.
    
    Let $\varepsilon>0$ be arbitrary. Starting with $k=1$,
    due to the properties of the function $\setf(\cdot)$,
    we can pick a $b$-adic interval $I_1^\varepsilon$ with left end point $0$ such that
    \begin{align}\label{eq:badicint}
        \int_{I_1^\varepsilon}G(x)\d x\leq\frac{M^\sobolev}{N}
        \leq\int_{I_1^\varepsilon}G(x)\d x+\frac{\varepsilon}{N}.
    \end{align}
    For $I_2^\varepsilon$, we set the left endpoint equal to the
    right endpoint of $I_1^\varepsilon$
    and choose the right endpoint of $I_2^\varepsilon$ as a $b$-adic knot
    such that \eqref{eq:badicint} is satisfied for $I_2^\varepsilon$.
    Repeating this procedure until $I_{N-1}^\varepsilon$ we get
    \begin{align*}
        \int_{\bigcup_{k=1}^{N-1}I_k^\varepsilon}G(x)\d x\leq
        \frac{N-1}{N}M^\sobolev\leq\int_{\bigcup_{k=1}^{N-1}I_k^\varepsilon}
        G(x)\d x+\frac{N-1}{N}\varepsilon.
    \end{align*}
    Taking $I_N^\varepsilon$ as the remaining interval such that
    $\bigcup_{k=1}^{N}I_k^\varepsilon=[0,1]$, we have
    \begin{align*}
        \int_{\bigcup_{k=1}^{N}I_k^\varepsilon}G(x)\d x=M^\sobolev.
    \end{align*}
    For the last interval we see that
    \begin{align*}
        \int_{I_N^\varepsilon}G(x)\d x\geq\frac{M^\sobolev}{N},
    \end{align*}
    and
      \begin{align*}
        \int_{I_N^\varepsilon}G(x)\d x &= M^\sobolev - \int_{\bigcup_{k=1}^{N-1}I_k^\varepsilon}G(x)\d x
        \leq  M^\sobolev - \frac{N-1}{N} \Big( M^\sobolev -  \varepsilon\Big)
        \leq \frac{1}{N}M^\sobolev+\varepsilon.
    \end{align*}
    
    Finally, we apply polynomial approximation over each $I_k^\varepsilon$.
    There exist polynomials $P_k$ of degree $\leq\m$ over each $I_k^\varepsilon$
    such that for $f_k:=f|_{I_k^\varepsilon}$
    (see \cite[Chapter 12, Theorem 8.1]{DeVore93})
    \begin{align*}
        \norm{f_k-P_k}[p]^\sobolev (I_k^\varepsilon)\leq
        C^\sobolev\snorm{f_k}[{\Bsqp[@][\sobolev][\sobolev]}]^\sobolev(I_k^\varepsilon)
        \leq C' \int_{I_k^\varepsilon}G(x)\d x\leq C'
        \begin{cases}
            \frac{1}{N}M^\sobolev,&\quad k=1,\ldots,N-1,\\
            \frac{1}{N}M^\sobolev+\varepsilon,&\quad k=N,
        \end{cases}
    \end{align*}
    where $\norm{\cdot}(I_k^\varepsilon)$ means we take norms
    over $I_k^\varepsilon$ only.
    Setting $s=\sum_{k=1}^NP_k\indicator{I_k^\varepsilon}$
    and since $p/\sobolev>1$, 
    we obtain
    \begin{align*}
        \norm{f-s}[p]^p&=\sum_{k=1}^N\norm{f_k-P_k}[p]^p(I_k^\varepsilon)\leq
        (N-1)CM^pN^{-p/\sobolev}+
        \left(\frac{1}{N}M^\sobolev+\varepsilon\right)^{p/\sobolev}\\
        &\leq(N-1)CM^pN^{-p/\sobolev}+2^{p/\sobolev-1}
        \left((\frac{1}{N}M^\sobolev)^{p/\sobolev}+
        \varepsilon^{p/\sobolev}\right)
        \leq \max\set{C, 2^{p/\sobolev-1}}\left(M^pN^{1-p/\sobolev}+
        \varepsilon^{p/\sobolev}\right).
    \end{align*}
    Since the constant is independent of $\varepsilon$ and
    $\varepsilon$ can be chosen arbitrarily small, we obtain
    \eqref{eq:badicfree}.
\end{proof}

\begin{proof}[Proof of \Cref{lemma:smallinterval}]
\label{proof:smallest}
    Let $f_k:=f\indicator{I_k}$. By the Hölder inequality
    \begin{align*}
        \norm{f_k}[p]^p=\int_0^1|f_k(x)|^p\d x\leq
        \left(\int_0^1|f(x)|^{p\delta}\d x\right)^{1/\delta}
        \left(\int_{I_k}\d x\right)^{1/q}.
    \end{align*}        
    We choose
    \begin{align*}
        \Lambda:=\set{k=1,\ldots,N:\;|I_k|>\varrho(\varepsilon)}.  
    \end{align*}
    Then,
    \begin{align*}
        \sum_{k\not\in\Lambda}\norm{f_k}[p]^p\leq
        \norm{f}[p\delta]^pN\varrho(\varepsilon)^{1/q}\leq\varepsilon^p.
    \end{align*}
    For $\tilde{s}$, we thus estimate
    \begin{align*}
        \norm{f-\tilde{s}}[p]^p=\sum_{k\in\Lambda}\norm{f_k-s_k}[p]^p
        +\sum_{k\not\in\Lambda}\norm{f_k}[p]^p\leq
        2\varepsilon^p.
    \end{align*}
\end{proof}

\begin{proof}[Proof of \Cref{thm:jacksonbesov}]
\label{proof:lemmafree}
    As in \Cref{thm:jacksonfixed}, we consider only the case
    $m+1\leq \alpha$, as the
    case $\alpha<m+1$ can be handled analogously with fewer steps.
    By \Cref{lemma:badicfree}, we can restrict ourselves to free knot splines
    with $b$-adic knots.
    By \Cref{lemma:smallinterval}, we can bound the size of the smallest interval
    and thus the level $d$.
    And finally, by \Cref{lem:boundranksfreeknotsplines},
    we can bound the ranks of an interpolation of a free knot spline.
    Thus, we have all the ingredients to bound
    the representation complexity of a free knot spline.
    It remains to combine these estimates with standard results
    from approximation theory to arrive at \eqref{eq:bestfreeknot}.
    
    Let $N\in\N$ be arbitrary.
    From \Cref{lemma:badicfree}, we know there exists a spline $s\in\Sfreeb{N}{\bar m}$ with $b$-adic knots such that
    \begin{align}\label{eq:errorfreeknot}
        \norm{f-s}[p]\leq
        C_1N^{-\alpha}\norm{f}[{\Bsqp[\alpha][\sobolev][\sobolev]}],
    \end{align}
    for some constant $C_1>0$.
    Set
    $\varepsilon:= 
    C_1\norm{f}[{\Bsqp[\alpha][\sobolev][\sobolev]}]N^{-\alpha}$.
    Since $\alpha>1/\sobolev-1/p$, there exists a $\delta>1$ such that
    $f\in\Lp[p\delta]$.
    By Lemma \ref{lemma:smallinterval}, 
    we can assume w.l.o.g. that $d := d(s)$ is such that
    $$
b^{-d} > N^{-q} \norm{f}[p\delta]^{pq} \varepsilon^{pq},
    $$
    or equivalently
    \begin{align*}
        d < q\log_b(\varepsilon^{-p}\norm{f}[p\delta]^p N )=
        q\log_b\left[C_1^{-p}\norm{f}[{\Bsqp[\alpha][\sobolev][\sobolev]}]^{-p}
       \norm{f}[p\delta]^{ {p}}N^{1+\alpha p}\right]\leq
       q\log_b\left[C_1^{-p}N^{1+\alpha p}\right],
    \end{align*}
    where $q=\delta/(\delta-1)$.
    
    We use the interpolant of \Cref{local-interpolation} and
    set $\tilde s:=\mc I_{b,\bar d,m}s$ for
    $\bar d\geq d$ to be specified later.
    Let $s_j:=\tensor{s}(j_1,\ldots,j_d,\cdot)$, where $\tensor{s} = T_{b,d}s$, 
    and analogously $\tilde s_j$.    
    For the re-interpolation error
    we can estimate similar to \Cref{lemma:reinterpol}
    \begin{align*}
        \norm{s-\tilde s}[p]^p&=\sum_{j\in I_b^d}b^{-d}\norm{s_j-\tilde s_j}[p]^p
        \leq C_2\sum_{j\in I_b^d}b^{-d}b^{-p(\bar d-d)(m+1)}
        \norm{s_j^{(m+1)}}[p]^p
        \leq C_3\sum_{j\in I_b^d}b^{-d}b^{-(\bar d-d)(m+1) {p}}\norm{s_j}[p]^p,
    \end{align*}
    where the latter follows from \cite[Theorem 2.7 of Chapter 4]{DeVore93},
    since $s_j$ is a polynomial of degree $\bar m$.

    Since $s$ is a quasi-interpolant of $f$,
    $s_j$ is a dilation of a polynomial (near-)best approximation of
    $f$ over the corresponding interval and thus
    by \cite[Theorem 8.1 of Chapter 12]{DeVore93}
    \begin{align*}
        \norm{s_j}[p]\leq C_4\snorm{f_j}[{\Bsqp[@][\sobolev][\sobolev]}]
    \end{align*}
    where $f_j:=\tensor{f}(j_1,\ldots,j_d,\cdot)$
    and for any $j\in I_b^d$.
    Together with \cite[Proposition 2.19]{partI},
    we finally estimate
    \begin{align*}
        \norm{s-\tilde s}[p]\leq C_5 b^{-(\bar d-d)(m+1)} b^{d/p}
        \snorm{f}[{\Bsqp[@][\sobolev][\sobolev]}].
    \end{align*}
  Thus,  to obtain at least 
    the same approximation order as in \eqref{eq:errorfreeknot},
    we set
    \begin{align*}
        \bar d:=\left\lceil\frac{d(m+1+1/p) + \alpha \log_b(N)}{m+1}\right\rceil \le C_6 \log_b(N),
    \end{align*}
    so that 
       \begin{align}\label{eq:freeknotreinterpol}
        \norm{s-\tilde s}[p]\leq C_5 N^{-\alpha} 
        \snorm{f}[{\Bsqp[@][\sobolev][\sobolev]}].
    \end{align}
  From \Cref{cor:freeknotinterpolant}, 
$\tilde{s}\in\Vbd{m}$ with
    \begin{align*}
        n:=\cost_{\mc C}(\tilde{s})&\leq C_7\left(N^2\log_b(N)+\log_b(N)\right)
        \leq C N^{2+\sigma},
    \end{align*}
    for any $\sigma>0$, where $C>0$ depends on $\sigma$.
    Similarly for $\cost_{\mc S}$ and $\cost_{\mc N}$, we obtain from \Cref{cor:freeknotinterpolant}
     that 
        \begin{align*}
        \cost_{\mc S}(\tilde{s})\leq\bar C(N\log_b(N)^2+\log_b(N))
        \leq CN^{1+\sigma},
    \end{align*}
    and 
    \begin{align*}
        \cost_{\mc N}(\tilde s)\leq \bar C(N\log_b(N)+\log_b(N))\leq CN^{1+\sigma},
    \end{align*}
    for any $\sigma>0$ and constants $C$ depending on $\sigma>0.$
    Combining \eqref{eq:errorfreeknot} with \eqref{eq:freeknotreinterpol},
    a triangle inequality
    and the above complexity bounds, we obtain the desired statement.
\end{proof}

\begin{proof}[Proof of \Cref{lemma:polys}]
\label{proof:polys}
    Let $P\in\P_{\bar{m}}$ be arbitrary
    and set $s:=\mc I_{b,d,m}P$.
    From \Cref{lemma:reinterpol},
    we obtain
    \begin{align}\label{eq:errorinterpol}
        \norm{P-s}[p]\leq C_1b^{-d(m+1)}\norm{P^{(m+1)}}[p].
    \end{align}
    From \Cref{ranks-polynomials} we can estimate the complexity of
    $s\in\Vbd{m}$ as
    \begin{align*}
        n:=\cost_{\mc C}(s)\leq b^2+b(d-1)(\m+1)^2+(m+1)^2,    
    \end{align*}
    or
    \begin{align*}
        d\geq \frac{n-b^2+ b(\m+1)^2-(m+1)^2}{b(\m+1)^2}.
    \end{align*}
    Inserting into \eqref{eq:errorinterpol}
    \begin{align*}
        \norm{P-s}[p]\leq C_2b^{-\frac{m+1}{b(\m+1)^2}n}\norm{P^{(m+1)}}[p].
    \end{align*}
    Analogously for $\cost_{\mc N}$
    \begin{align*}
        n:=\cost_{\mc N}(s)\leq d(\m+1),
    \end{align*}
    and
    \begin{align*}
        \norm{P-s}[p]\leq C_2b^{-\frac{m+1}{(\m+1)}n}\norm{P^{(m+1)}}[p].
    \end{align*}
\end{proof}

\begin{proof}[Proof of \Cref{thm:analytic}]
\label{proof:analytic}
    Set
    \begin{align*}
        M:=\sup_{z\in D_\rho}|f(z)|,
    \end{align*}
    and $\bar{m}\in\N$.
    From \cite[Chapter 7, Theorem 8.1]{DeVore93}, we know
    \begin{align}\label{eq:erroranal}
        \inf_{P\in\P_{\bar{m}}}\norm{f-P}[\infty]\leq
        \frac{2M}{\rho-1}\rho^{-\bar{m}}.
    \end{align}
    We aim at approximating an arbitrary polynomial of degree $\bar{m}$
    within $\Vb{m}$. W.l.o.g.\ we can assume $\bar{m}>m$,
    since otherwise $\P_{\bar{m}}\subset\Vb{m}$.
    
    From \eqref{eq:errorinterpol} we know
    \begin{align}\label{approx-error-infini-P}
        \norm{P-s}[\infty]\leq C_1b^{-d(m+1)}
        \norm{P^{(m+1)}}[\infty],
    \end{align}
    for a spline $s = \mathcal{I}_{b,d,m} P$ of degree $m$.    
    To estimate the derivatives $\norm{P^{(m+1)}}[\infty]$, we further
    specify $P$. Let $P$ be the sum of
    Chebyshev polynomials from \cite[Chapter 7, Theorem 8.1]{DeVore93} used to derive
    \eqref{eq:erroranal}. I.e., since $f$ is assumed to be analytic, we
    can expand $f$ into a series
    \begin{align*}
        f(x)=\frac{1}{2}a_0+\sum_{k=1}^\infty a_k C_k(x),
    \end{align*}
    where $C_k$ are Chebyshev polynomials of the first kind of
    degree $k$. We set $P=\Pch$ with 
    \begin{align}\label{eq:cheb}
         \Pch:=\frac{1}{2}a_0+\sum_{k=1}^{\bar{m}}a_kC_k,
    \end{align}
    which is such that 
    \begin{equation}
    \Vert f - \Pch \Vert_p \le  \frac{2M}{\rho-1}\rho^{-\bar{m}}. \label{error-Chebyshev}
     \end{equation}
    For the derivatives of $C_k^{(m+1)}$ we get by standard estimates
    (see, e.g., \cite{Mason2002})
    \begin{align*}
        \norm{C_k^{(m+1)}}[\infty]
        \leq\frac{k^2(k^2-1)\cdots(k^2-m^2)}{(2(m+1)-1)!}.
    \end{align*}
    And thus, for any $1<\rho_0<\rho$,
    \begin{align*}
        \norm{\Pch^{(m+1)}}[\infty]&\leq\frac{1}{(2(m+1)-1)!}
        \sum_{k=m+1}^{\bar{m}}|a_k|k^2(k^2-1)\cdots(k^2-m^2)\\
        &\leq\frac{2M}{(2(m+1)-1)!}
        \sum_{k=m+1}^{\bar{m}}\rho_0^{-k}k^2(k^2-1)\cdots(k^2-m^2).
    \end{align*}
    For $\bar{m}\rightarrow\infty$, this series is a sum of
    poly-logarithmic series,
    and thus it converges to a constant depending on $M$, $m$ and $\rho$.
    
    We can now combine both estimates for the final approximation error. We first consider the approximation error 
    $E_n^{\mathcal{C}}(f)_\infty$. 
    Let $n\in\N$ be large enough such that
    \begin{align*}
        d&:=\left\lfloor b^{-1}n^{1/3}- (m+1) {n^{-2/3}}\right\rfloor>1,\\
        \bar{m}&:=\lfloor n^{1/3}-1\rfloor\geq 1.
    \end{align*}
    For this choice of $d$ and $\bar{m}$,
    let $s\in\Vbd{m}$ be the interpolant of degree $m$ of the Chebyshev polynomial
    $\Pch$ from \eqref{eq:cheb}.
    Then from \Cref{prop:encoding-poly-interpolant}, we obtain
    \begin{align*}
        \cost_{\mc C}(s)\leq bd(\bar{m}+1)^2+ b(m+1)\leq n,
    \end{align*}
    and thus $s\in\Phi_n$.
    Moreover, by \eqref{approx-error-infini-P} and \eqref{error-Chebyshev},
    \begin{align*}
       E_n^{\mathcal{C}}(f)_\infty&\leq\norm{f-\Pch}[\infty]+\norm{\Pch-s}[\infty]
        \leq\frac{2M}{\rho-1}\rho^{-\bar{m}}+
        C'_1b^{-d(m+1)}\norm{\Pch^{(m+1)}}[\infty]\\
        &\leq C'_2[\min(\rho,b^{{(m+1)/}b})]^{-n^{1/3}}.
    \end{align*}
    The result for $E_n^{\mathcal{S}}(f)_\infty$ follows from $\tool^{\mathcal{C}}\subset \tool^{\mathcal{S}}.$
    Now we consider the case  of  $E_n^{\mc N}(f)_\infty$. Let $n\in \N$ be large enough such that 
    $d:= \lfloor n^{1/2} \rfloor>1$ and $\bar m = \lfloor n^{1/2} -1 \rfloor\ge 1$. Then from \Cref{prop:encoding-poly-interpolant}, we obtain
    $$
    \cost_{\mc N}(s) \le d(\bar m+1) \le n.
    $$
    Moreover, by \eqref{approx-error-infini-P} and \eqref{error-Chebyshev},
        \begin{align*}
       E_n^{\mathcal{N}}(f)_\infty&\leq 
         \frac{2M}{\rho-1}\rho^{-\bar{m}}+
        C_1'b^{-d(m+1)}\norm{\Pch^{(m+1)}}[\infty]
        \leq C_2'[\min(\rho,b^{(m+1)})]^{-n^{1/2}}.
    \end{align*}
    \end{proof}

\end{document}